\documentclass[leqno,final]{siamltex}
\usepackage{graphicx} 
\usepackage{amsmath,amstext,amssymb,bm}
\usepackage{leftidx}
\usepackage{booktabs} 
\usepackage[table]{xcolor}
\usepackage{xcolor} 
\usepackage{soul} 
\usepackage{tikz}
\usetikzlibrary{shapes,arrows}
\usepackage{mathrsfs}
\usepackage{epstopdf}
\usepackage{color}
\usepackage{multirow}
\usepackage{tabularx}
\usepackage[shortlabels]{enumitem}

\numberwithin{equation}{section}
\newtheorem{remark}{Remark}[section]

\allowdisplaybreaks[4]

\usepackage{hyperref}
\hypersetup{
	colorlinks=true,   
	linkcolor=blue,    
	citecolor=blue,    
	filecolor=magenta, 
	urlcolor=cyan      
}

\def\div{\mbox{div }}

\def\u{\mathbf{u}}
\def\e{\mathbf{e}}

\def\v{\textbf{v}}
\def\w{\textbf{w}}
\def\f{\mathbf{f}}

\def\V{\textbf{V}}

\def\c{\textbf{c}}

\def\n{\textbf{n}}

\def\P{\textbf{P}}

\def\i{\textbf{i}}
\def\A{\textbf{A}}

\renewcommand{\div}{\mbox{\rm div\,}}

\begin{document}
	
	\title{ A decoupled Crank-Nicolson leap-frog scheme for the unsteady bioconvection flows problem with concentration dependent viscosity
		}
	\markboth{CHENYANG LI
		}{Fully discretization for Bioconvection}

	\author{Chenyang Li\thanks{School of Mathematical Sciences, East China Normal University, Shanghai, P. R. China. (chenyangli1004@yeah.net).}
	}
	
	\maketitle
	\begin{abstract} 
A fully discrete Crank--Nicolson Leap--Frog (CNLF) scheme is proposed and analyzed for the unsteady bioconvection flow problem with concentration-dependent viscosity. Spatial discretization is handled via the Galerkin finite element method (FEM), while temporal discretization employs the CNLF method for the linear terms and a semi-implicit approach for the nonlinear terms. The scheme is proven to be unconditionally stable, i.e., the time step is not subject to a restrictive upper bound. Using the energy method, $L^2$-optimal error estimates are derived for the velocity and concentration . Finally, numerical experiments are presented to validate the theoretical results.
	\end{abstract}
	
	
	
	\section{Introduction}
	The bioconvection model is coupled by the Navier-Stokes
	type equations describe the flow of the incompressible viscous culture fluid and the advection-diffusion equations describe the transport of micro-organisms:
	\begin{align}
		\frac{\partial \u}{\partial t} - \div ( \nu(c) D(\u)) + \u \cdot \nabla \u + \nabla p = - g (1+\gamma c) \textbf{i}_2 + \textbf{f}, \quad  x \in \Omega, \, t>0,\label{biobdf-32}\\
		\nabla \cdot \u = 0, \quad  x \in \Omega, \, t>0, \\
		\frac{\partial c}{\partial t} -  \theta \Delta c + \u \cdot \nabla c + U \frac{\partial c}{\partial x_2}=0, \quad  x \in \Omega, \, t>0.\label{biobdf-33}
	\end{align}
	
	The unknowns in the bioconvection model are the concentration $c$, the velocity $\mathbf{u}$, and the pressure $p$, with the latter assumed to have zero mean for uniqueness. 
	Here, $\Omega \subset \mathbb{R}^d$ ($d=2$ or $3$) is a bounded domain with smooth boundary $\partial \Omega$, and $x_2$ denotes the second component of the spatial variable $x$. 
	The parameter $\theta>0$ represents the diffusivity, and the kinematic viscosity $\nu(c)$ depends on the concentration of the micro-organisms \cite{batchelor1972,brady1993}. 
	The stress tensor is defined as $D(\mathbf{u}) = \frac{1}{2} (\nabla \mathbf{u} + \nabla \mathbf{u}^{T})$, $\mathbf{f}$ denotes the external force, $g$ is the gravitational acceleration, and $U>0$ represents the mean upward swimming velocity of the micro-organisms.The parameter $\gamma>0$ characterizes the relative difference between the density $\rho_0$ of the micro-organisms and the density $\rho_m$ of the suspending fluid, defined by $\gamma = \frac{\rho_0}{\rho_m}-1$. 
	The term $-g (1+\gamma c) \mathbf{i}_2$ accounts for the gravitational force acting on the micro-organisms, while the term $U \frac{\partial c}{\partial x_2}$ models the effect of their average upward swimming.
	
	In classical Newtonian fluids, the viscosity is constant; however, this assumption is not valid for real suspensions, where the viscosity depends on the micro-organism concentration. 
	Exponent-type expressions of $\nu(c)$ as a function of $c$ have been extensively discussed in \cite{batchelor1972,brady1993,mooney1951,krieger1959}. 
	
We impose homogeneous Dirichlet boundary conditions for the velocity, $\mathbf{u}=0$, and a non-flux Robin boundary condition for the concentration:
\begin{align*}
	\theta \frac{\partial c}{\partial \n} - c U n_2 = 0,
\end{align*}
where $\mathbf{n} = (n_1,n_2)$ denotes the outward unit normal vector on $\partial \Omega$. 
Furthermore, we assume that the total mass of micro-organisms in the container is conserved, i.e.,
\begin{align}\label{biobdf-1}
	\frac{1}{|\Omega|} \int_{\Omega} c(x)\, dx = \alpha,
\end{align}
where $|\Omega|$ denotes the measure of $\Omega$ and $\alpha$ represents the average concentration.

Finally, the governing equations of bioconvection are prescribed on a bounded time interval $I=(0,T]$ with the following initial and boundary conditions:
\begin{align}\label{biobdf-45}
	\left\{\begin{aligned}
		\frac{\partial \u}{\partial t} - \div(\nu(c) D (\u)) + \u \cdot \nabla \u + \nabla p = - g (1+\gamma c) \textbf{i}_2 + \textbf{f}, \quad  &x \in \Omega, \, t>0,\\
		\nabla \cdot \u = 0, \quad  &x \in \Omega, \, t>0, \\
		\frac{\partial c}{\partial t} -\theta \Delta c + \u \cdot \nabla c + U \frac{\partial c}{\partial x_2}=0, \quad  &x \in \Omega, \, t>0,\\
		\frac{1}{|\Omega|} \int_{\Omega} c(x) dx = \alpha, \quad &x \in \Omega, \, t>0.\\
		\u=0, \,\, \theta \frac{\partial c}{\partial n} - c U n_2=0, \quad &x \in \partial \Omega, \, t>0.
	\end{aligned}\right.
\end{align}

This system models the coupled dynamics of micro-organism transport and fluid flow under the influence of gravity and chemotactic swimming, accounting for concentration-dependent viscosity and mass conservation.

The well-posedness of solutions for both time-dependent and steady bioconvection flow problems has been the subject of extensive study. For the system (\ref{biobdf-45}) with constant viscosity, the existence of solutions was established using the semi-group approach in \cite{kan1992}, and corresponding numerical investigations were reported in \cite{harashima1988}. In \cite{climent2013}, the authors proved the existence and uniqueness of periodic solutions for bioconvection flows with concentration-dependent density. More recently, under more general boundary conditions, \cite{cao2020} established the existence and uniqueness of weak solutions for system (\ref{biobdf-45}) under relatively mild and physically realistic assumptions on the viscosity, in particular the uniform boundedness condition $	\kappa \le \nu(x) \le \kappa^{-1}$. These results provide a rigorous theoretical foundation for the numerical analysis of bioconvection problems with non-constant viscosity and support the development of finite element schemes for such systems.

Considerable attention has also been devoted to the numerical analysis and simulation of bioconvection flow problems. 
Finite element approximations using stable velocity--pressure pairs have been developed, and $H^1$-norm error estimates for both velocity and concentration were established in \cite{cao2014}. 
Spectral Galerkin approximations for unsteady bioconvection flows and their convergence rates were investigated in \cite{aguiar2017}. 
A fully discrete finite element scheme based on the backward Euler method was proposed in \cite{cao2020}, where the corresponding $L^2$-norm error estimates were shown to be sub-optimal. 
For the steady bioconvection problem, \cite{colmenares2021} reformulated the system as a first-order formulation by introducing the shear-stress, vorticity, and pseudo-stress tensors in the fluid equations, together with an auxiliary vector in the concentration equation. 
Existence and uniqueness results were then established using the Lax--Milgram theorem and Banach fixed-point theory, and an augmented mixed finite element method was analyzed. 
Recently, both coupled and decoupled BDF2 finite element schemes have been derived in \cite{lichenyangbio}, providing optimal-order error estimates for velocity and concentration in both $L^2$- and $H^1$-norms on bounded domains. 
Moreover, by employing Stokes projection operators with concentration-dependent viscosity coefficients and leveraging mathematical induction, a linearized second-order Crank--Nicolson finite element scheme for the bioconvection model was developed in \cite{liyuanbio2025}, demonstrating enhanced temporal accuracy while maintaining rigorous error control.
Collectively, these studies lay a numerical foundation for the development of higher-order, stable, and efficient finite element schemes for both time-dependent and steady bioconvection problems, especially in the presence of concentration-dependent viscosity and complex boundary conditions.

 The objective of this paper is to introduce the CNLF fully discrete finite element scheme for solving the time-dependent bioconvection flows problem with concentration dependent viscosity. We employ the mini element (P1b--P1) for the approximation of velocity and pressure, and the piecewise linear (P1) element for the approximation of the concentration field.The fully  decoupled CNLF finite element method are proposed, and the unconditionally optimal convergent rate for the velocity and concentration in $L^2$-norm and $H^1$-norm are established. This semi-implicit treatment not only leads to a linear system at each time step, 
 but also plays a crucial role in maintaining the nearly unconditional stability 
 and preserving the convergence order of the corresponding nonlinear implicit schemes.

The remainder of this paper is organized as follows. 
In Section~\ref{sec-2}, we introduce the function and space notations, along with several preliminary results that will be used throughout the paper. 
We then reformulate the bioconvection problem and present the variational  formulation of the solution.  
In Section~\ref{sec-3}, we introduce the Crank--Nicolson Leap--Frog (CNLF) finite element approximation for the bioconvection model and derive the unconditional stability results. 
Section~\ref{sec-4}, which constitutes the main part of this work, is devoted to a rigorous analysis of the fully discrete CNLF finite element scheme using energy methods. 
The highlights of this section include the derivation of optimal-order error estimates in the energy norm for both velocity and concentration approximations. 
In Section~\ref{sec-5}, the theoretical results are validated through numerical experiments. 
Different viscosity models are employed to investigate the convergence behavior and to demonstrate the effectiveness of the proposed scheme. 

Throughout this paper, the symbol $C$ denotes a generic positive constant whose value may differ from one occurrence to another, but which is independent of the discretization parameters, namely, the time step size $\tau$ and the mesh size $h$.

	\section{Preliminaries and  useful facts}\label{sec-2}
Throughout this paper, we adopt the standard notation for vector-valued Sobolev spaces.
	\begin{align}
		\textbf{V}&=H^1_0 = \{  \v \in H^1(\Omega) ; \,\, \v |_{\partial \Omega} =0 \},\\
		\textbf{V}_0 &= \{ \v \in \textbf{V}; \,  \nabla \cdot \v=0 \,\, \text{in} \,  \Omega\},\\
		M &= L^2_0 (\Omega) = \{ q \in L^2(\Omega); \,\, \int_{\Omega} q dx =0    \},\\
		\tilde{H}&=H^1(\Omega)  \cap L^2_0(\Omega). 
	\end{align}

For $k\in N^+$ and $1\leq p\leq +\infty$, we denote $L^p(\Omega)$ and $W^{k, p}(\Omega)$ as the classical Lebesgue space and Sobolev space, respectively. The norms of these spaces are denoted by
\begin{align*}
	||u||_{L^p(\Omega)}&=\left(\int_{\Omega}|u(\mathbf{x})|^p d \mathbf{x}\right)^\frac{1}{p},\\
	||u||_{W^{k,p}(\Omega)}&=\left(\sum\limits_{|j|\leq k}||D^ju||_{L^p(\Omega)}^p\right)^\frac{1}{p}.
\end{align*}
within this context, $W^{k, 2}(\Omega)$ is also known as the Hilbert space and can be expressed as $H^k(\Omega)$.  $||\cdot||_{L^\infty}$ represents the norm of the space  $L^\infty(\Omega)$ which is defined as
\begin{equation*}
	||u||_{L^\infty(\Omega)}=ess\sup\limits_{\mathbf{x}\in \Omega}|u(\mathbf{x})|.
\end{equation*}

For any Hilbert space $D$, where $T>0$, with the corresponding continuous and discrete norms, respectively,
\[
\begin{aligned}
	L^p(0,T;D) &= 
	\left\{
	v : [0,T] \to D \ \Bigg| \
	\|v\|_{l^p(0,T;D)} =
	\left[ \int_0^T \|v(t)\|_D^p \, dt \right]^{1/p} < \infty
	\right\},
	\quad 1 \le p < \infty, \\[1em]
	L^\infty(0,T;D) &=
	\left\{
	v : [0,T] \to D \ \Bigg| \
	\|v\|_{l^\infty(0,T;D)} = 
	\operatorname*{ess\,sup}_{t \in [0,T]} \|v(t)\|_D < \infty
	\right\}, \\[1em]
	l^p(0,T;D) &=
	\left\{
	v : \{t_1, \cdots, t_M\} \to D \ \Bigg| \
	\|v\|_{L^p(0,T;D)} =
	\left[ \Delta t \sum_{i=1}^M \|v(t_i)\|_D^p \right]^{1/p} < \infty
	\right\},
	\quad 1 \le p < \infty, \\[1em]
	l^\infty(0,T;D) &=
	\left\{
	v : \{t_1, \cdots, t_M\} \to D \ \Bigg| \
	\|v\|_{L^\infty(0,T;D)} = 
	\max_{1 \le i \le M} \|v(t_i)\|_D < \infty
	\right\},
\end{aligned}
\]
where $D = L^2(\Omega), \V,$ or $W$.

For simplicity, we denote the inner products of both
$L^2(\Omega)$ and $\textbf{L}^2(\Omega)$
by $(\cdot,\cdot)$, and use $\langle \cdot,\cdot \rangle$ to denote the dual product of $H^{-1}(\Omega) \times H^1_0(\Omega)$. namely,
\begin{align*}
	\begin{split}
		&(u,v)=\int_\varOmega u(x)v(x) d x \quad \forall  \, u,v\in L^2(\Omega),\\
		&({\bf u},{\bf v})=\int_\Omega {\bf u}(x)\cdot {\bf v}(x) dx \quad \forall \, {\bf u},{\bf v}\in {\bf L}^2(\Omega) .\\
	\end{split}
\end{align*}

Denote 
\begin{align}
	\A(c,\u,\v)=(\nu(c+\alpha) \nabla \u,\nabla \v), \quad \forall \, c \in \tilde{H},\,  \u,\v \in \V.
\end{align}

To prove the unconditional stability and error estimate of the following spatial discrete schemes, we recall the following discrete Gronwall inequality established in \cite{evance2022,heywood1990}.
\begin{lemma} (Discirete Gronwall's inequality ) Let $a_k , b_k$ and $y_k$ be the nonnegative numbers such that \label{biobdf-11}
	\begin{align}\label{growninequality-discrete}
		a_n+ \tau \sum \limits^n \limits_{k=0}  b_k\leq \tau \sum \limits^n \limits_{k=0} \gamma _k a_k + B \quad \text{for} \,\, n \geq 1,
	\end{align}
	Suppose $\tau \gamma _k \leq 1$ and set $\sigma_k = (1-\tau \gamma_k) ^{-1}$. Then there holds
	\begin{align}\label{grown2}
		a_n + \tau \sum \limits^n \limits_{k=0} b_k \leq exp(\tau \sum \limits^n \limits_{k=0} \gamma_k \sigma_k) B \quad  \text{for} \,\, n\geq 1.
	\end{align}
	\begin{remark}
		If the sum on the right-hand side of (\ref{growninequality-discrete}) extends only up to $n-1$, then the estimate (\ref{grown2}) still holds for all $k \geq 1$ with $\sigma_k=1$.
	\end{remark}
\end{lemma}

The following regularity assumption also plays an essential role in the stability and convergence analysis.

\textbf{Assumption A1}:
The viscosity function $\nu(c)$ is Lipschitz continuous, and there exist positive constants 
$\lambda, \beta > 0$ such that
\begin{equation}\label{biocnlf-assumption1}
	\kappa \le \nu(x) \le \kappa^{-1}, \qquad 
	|\nu(x_1) - \nu(x_2)| \le \beta |x_1 - x_2|, 
	\quad \forall\, x, x_1, x_2 \in \mathbb{R}.
\end{equation}

\textbf{Assumption A2}
Assume that the exact solutions and the body force satisfy the following regularity conditions:
\begin{align}
	&\f \in L^2(0,T;L^2(\Omega))\\
&\u\in L^{\infty} (0,T;W^{2,4}(\Omega)),\quad \u_t\in L^2(0,T;L^2(\Omega)),\quad\u_{tt}\in L^2(0,T;H^1(\Omega)), \quad	\u_{ttt}\in L^2(0,T;L^2(\Omega)),\\
&c\in L^{\infty} (0,T;H^2(\Omega)),\quad c_t\in L^2(0,T;L^2(\Omega)),\quad c_{tt}\in L^2(0,T;H^1(\Omega)), \quad	c_{ttt}\in L^2(0,T;L^2(\Omega)).
\end{align}

The transport terms present a difficulty since the corresponding discrete formulations do not preserve the skew-symmetry  property as in the continuous case. To overcome this issue, we introduce appropriately defined skew-symmetric trilinear forms, which facilitate the stability analysis and the derivation of error estimates.
\begin{align}
	\begin{split}
		B(\textbf{u},\textbf{v},\textbf{w}) &= \int_\Omega (\u \cdot \nabla \textbf{v}) \cdot\textbf{w} dx + \frac{1}{2} \int_{\Omega} (\nabla \cdot \u) \v \cdot \w dx \\
		&=   \frac{1}{2}\int_\Omega (\u \cdot \nabla \textbf{v}) \cdot\textbf{w} dx -  \frac{1}{2} \int_{\Omega} (\u \cdot \nabla) \w \cdot\v dx  \quad \forall ~\u,\v,\w\in\textbf{V},\\
		b(\u,c,r) &=  \int_\Omega (\u \cdot \nabla c)r dx+ \frac{1}{2} \int_{\Omega} (\nabla \cdot \u) crdx, \\
		&=\frac{1}{2}\int_\Omega (\u \cdot \nabla c) r dx -  \frac{1}{2} \int_{\Omega} (\u \cdot \nabla r) c dx \quad \forall ~\u \in \V, \,  \forall c,r \in \tilde{H}.
	\end{split}
\end{align}
which has the following properties \cite{heyinnian2005,heyinnian2022}
\begin{align}
	& B(\textbf{u},\textbf{v},\textbf{v})=0, \quad b(\u,r,r)=0,\label{biobdf-2}\\
	&B(\textbf{u},\textbf{v},\textbf{w}) = - b(\textbf{u},\textbf{w},\textbf{v}),\quad b(\textbf{u},c,r) = - b(\textbf{u},r,c),\label{biobdf-17}\\
	&B(\textbf{u},\textbf{v},\textbf{w}) \leq C \| \nabla \textbf{u} \|_{L^2} \| \nabla \textbf{v} \| _{L^2} \| \nabla \textbf{w} \|_{L^2},\quad b(\textbf{u},c,r) \leq C \| \nabla \textbf{u} \|_{L^2} \| \nabla c \| _{L^2} \| \nabla r \|_{L^2},\\
	&B(\textbf{u},\textbf{v},\textbf{w})  \leq C \| \textbf{u} \| ^{\frac{1}{2}}_{L^2} \| \nabla \textbf{u}\| ^{\frac{1}{2}}_{L^2} \| \nabla \textbf{v} \| _{L^2} \| \nabla \textbf{w} \|_{L^2}, \quad  b(\textbf{u},c,r)  \leq C \| \textbf{u} \| ^{\frac{1}{2}}_{L^2} \| \nabla \textbf{u}\| ^{\frac{1}{2}}_{L^2} \| \nabla c \| _{L^2} \| \nabla r \|_{L^2}.
\end{align}

If $\nabla \cdot \u=0$, there holds $B(\textbf{u},\textbf{v},\textbf{w}) = (\textbf{u} \cdot \nabla \textbf{v},\textbf{w})= \int_\Omega (\u \cdot \nabla \textbf{v}) \cdot\textbf{w} dx $ and $b(\u,c,r) = (\textbf{u} \cdot \nabla c,r) = \int_\Omega (\u \cdot \nabla c)r dx$.

The following Sobolev embedding inequalities in 2D will be used in the following \cite{reference1}
\begin{align}
	W^{2,4}(\Omega) &\hookrightarrow W^{1,\infty}(\Omega), \label{biocnlf-w24}\\
	H^2(\Omega) &\hookrightarrow W^{1,q}(\Omega), \quad 2 \leq q < \infty, \label{biocnlf-h2w100}\\
	H^2(\Omega) &\hookrightarrow L^{\infty}(\Omega).
\end{align}

The discrete divergence-free velocity space is defined by
\begin{align}\label{divergence-free}
	\mathbf{V}_{0h} := \{ \mathbf{v}_h \in \mathbf{V}_h : (\nabla \cdot \mathbf{v}_h, q_h) = 0, \quad \forall q_h \in M_h \}.
\end{align}

	
	Noticing that (\ref{biobdf-1}) is equivalent to require that $c-\alpha \in L^2_0(\Omega)$, we can adopt the same method in \cite{cao2020,antwolevel2023,antwolevel2024}, introducing an auxiliary concentration $c_\alpha=c-\alpha$ and $\textbf{f}_\alpha=\textbf{f}-g\gamma \alpha \textbf{i}_2$, then we can rewrite the original form (\ref{biobdf-45}) to the following system (still denote $c=c_\alpha, \textbf{f} = \textbf{f}_\alpha$ to simplify notation).
	\begin{align}\label{biobdf-3}
		\left\{\begin{aligned}
			\frac{\partial \u}{\partial t} - \div ( \nu(c+\alpha) \nabla \u) + \u \cdot \nabla \u + \nabla p = - g (1+\gamma c) \textbf{i}_2 + \textbf{f}, \quad  &x \in \Omega, \, t>0,\\
			\nabla \cdot \u = 0, \quad  &x \in \Omega, \, t>0, \\
			\frac{\partial c}{\partial t} - \theta  \Delta c + \u \cdot \nabla c + U \frac{\partial (c+\alpha)}{\partial x_2}=0, \quad  &x \in \Omega, \, t>0,\\
			\frac{1}{|\Omega|} \int_{\Omega} c(x) dx = 0, \quad &x \in \Omega, \, t>0.\\
			\u=0, \,\, \theta \frac{\partial c}{\partial n} - U(c+\alpha)n_2=0, \quad &x \in \partial \Omega, \, t>0.
		\end{aligned}\right.
	\end{align}

 The weak formulation of the bio-convection model (\ref{biobdf-3}) read as follows. We find $(\u,p,c) \in \textbf{V} \times M\times \tilde{H}$ such that  
\begin{align}\label{biobdf-4}
	\left\{\begin{aligned}
		&(\frac{\partial \u}{\partial t} , \v) + (\nu(c+\alpha) \nabla \u,\nabla \v) + B(\u,\u,\v) - ( \div \v,p) + (\div \u,q) = (\f,\v)- g((1+\gamma c)\i_2,\v),\\
		&(\frac{\partial c}{\partial t},\phi) +\theta (\nabla c, \nabla \phi) +b(\u,c,\phi) -U(c,\frac{\partial \phi}{\partial x_2}) = U\alpha (1, \frac{\partial \phi}{\partial x_2}),\\
		&\u(0)=\u_0, \quad c(0)=c_0.
	\end{aligned}\right.
\end{align}
for any $(\v,q,r)\in \V \times M\times \tilde{H}$.

\section{Decoupled Crank--Nicolson Leap-Frog FEM Algorithm and unconditional stability}\label{sec-3}

In this subsection, we present the decoupled Crank--Nicolson Leap--Frog (CNLF) scheme for the bioconvection model and derive its stability properties by means of the energy method.

Let $N$ be a positive integer and $0 = t_0 < t_1< \cdots <t_N=T$ be a uniform partition of $[0,T]$, with $\tau = \Delta t= T/N $. Let $\mathcal{T}_h$ be a family of quasi-uniform triangulations of the domain $\Omega$. 
The ordered elements are denoted by $\mathcal{K}_1, \mathcal{K}_2, \ldots, \mathcal{K}_n$, where $h_i = \operatorname{diam}(\mathcal{K}_i)$ for $i=1,2,\ldots,n$, and we define $h = \max\{h_1,h_2,\ldots,h_n\}$. 
For each $\mathcal{K} \in \mathcal{T}_h$, let $P_r(\mathcal{K})$ denote the space of all polynomials on $\mathcal{K}$ of degree at most $r$. 

We employ the mini element (P1b--P1) to approximate the velocity and pressure, and the piecewise linear (P1) element to approximate the concentration. 
The corresponding finite element spaces are defined as follows \cite{antwolevel2023, scott2008}:
\begin{align*}
	\V_h &= \left\{ \mathbf{v}_h \in C(\overline{\Omega})^2 \cap \mathbf{V} \; \big| \;
	\mathbf{v}_h|_{\mathcal{K}} \in \left( P_1(\mathcal{K}) \oplus b(\mathcal{K}) \right)^2, \ \forall \mathcal{K} \in \mathcal{T}_h \right\},\\
	M_h &= \left\{ q_h \in C(\overline{\Omega}) \cap H^1(\Omega) \; \big| \;
	q_h|_{\mathcal{K}} \in P_1(\mathcal{K}), \ \forall \mathcal{K} \in \mathcal{T}_h, \ 
	\int_{\Omega} q_h \, dx = 0 \right\},\\
	\widetilde{H}_h &= \left\{ r_h \in C(\overline{\Omega}) \cap \widetilde{H} \; \big| \;
	r_h|_{\mathcal{K}} \in P_1(\mathcal{K}), \ \forall \mathcal{K} \in \mathcal{T}_h \right\},
\end{align*}
where $b(\mathcal{K})$ denotes the bubble function associated with each element $\mathcal{K} \in \mathcal{T}_h$.

It is well known that the pair $(\V_h, M_h)$ satisfies the discrete LBB (Ladyzhenskaya--Babuska--Brezzi) condition \cite{babuvska1973, brezzi1974} for the mini element; that is, there exists a constant $\beta>0$, independent of the mesh size $h$, such that
\begin{align*}
	\beta \|q_h \|_{L^2} 
	\leq  \sup_{\mathbf{v}_h \in \V_h}  
	\frac{(\nabla \cdot \mathbf{v}_h,\, q_h)}{\|\nabla \mathbf{v}_h\|_{L^2}}, 
	\quad \forall q_h \in M_h.
\end{align*}

The following inverse inequalities will be used frequently:
\begin{align}
	\|\mathbf{u}_h\|_{W^{m,q}} 
	&\leq C h^{\,l-m + n\left(\frac{1}{q} - \frac{1}{p}\right)} 
	\|\mathbf{u}_h\|_{W^{l,p}}, 
	\quad \forall\, \mathbf{u}_h \in \V_h, \label{biocnlf-inverse1}\\
	\|c_h\|_{W^{m,q}} 
	&\leq C h^{\,l-m + n\left(\frac{1}{q} - \frac{1}{p}\right)} 
	\|c_h\|_{W^{l,p}}, 
	\quad \forall\, c_h \in \widetilde{H}_h. \label{biocnlf-inverse2}
\end{align}

Let the initial approximations be defined as $\mathbf{u}_h^0 = \mathcal{I}_h \mathbf{u}_0$ and $c_h^0 = \Pi_h c_0$, where $\mathcal{I}_h$ and $\Pi_h$ denote the standard interpolation operators onto $\V_h$ and $\widetilde{H}_h$, respectively. 
Then the following estimates hold \cite{scott2008}:
\begin{align}
	\|\mathbf{u}_0 - \mathbf{u}_h^0\|_{L^2} 
	+ \|c_0 - c_h^0\|_{L^2} &\leq C h^2, \label{biobdf-26}\\
	\|\mathbf{u}_0 - \mathbf{u}_h^0\|_{H^1} 
	+ \|c_0 - c_h^0\|_{H^1} &\leq C h, \label{biobdf-27}\\
	\|\mathbf{u}_h^0\|_{L^{\infty}} 
	+ \|c_h^0\|_{L^{\infty}} &\leq C.
\end{align}

Moreover, the following interpolation error estimates hold:
\begin{align}
	\begin{split}
		\|\mathbf{u} - \mathcal{I}_h \mathbf{u}\|_{L^2} 
		+ h \|\mathbf{u} - \mathcal{I}_h \mathbf{u}\|_{H^1} 
		&\leq C h^2 \|\mathbf{u}\|_{H^2},\\
		\|p - J_h p\|_{L^2} 
		&\leq C h \|p\|_{H^1},\\
		\|c - \Pi_h c\|_{L^2} 
		+ h \|c - \Pi_h c\|_{H^1} 
		&\leq C h^2 \|c\|_{H^2},
	\end{split}
\end{align}
where $J_h$ denotes the classical interpolation operator from $M$ onto $M_h$.

We introduce a projection operator with variable coefficients onto finite element spaces, For $0\leq n\leq N-1$, we define the stokes projection operator $(\P_h^{n+1}, \rho_h^{n+1}) : \V \times M_h \rightarrow \V_h \times M_h$ with variable coefficient by \cite{liyuanmhd2024}
\begin{align}
	\nu(c) (\nabla(\u -\P_h^{n+1} \u ), \nabla \v_h   ) +( \nabla \cdot \v_h, p-\rho_h^{n+1} p)=0, \quad \forall \v_h \in \V_h,\label{biocnlf-projection1}\\
	( \nabla \cdot (\u - \P_h^{n+1} \u),q_h   ) =0, \quad \forall q_h \in M_h,
\end{align}
and there holds 
\begin{align}
	\| \u - P_h^{n+1} \u \|_{L^2} + h ( \| \nabla (\u -\P_h^{n+1} \u)\|_{L^2} + \| p -\rho_h^{n+1} p\|_{L^2} \leq C h^2    (   \| \u\|_{H^2}  + \| p\|_{H^1} ).\label{biocnlf-projection-u}
\end{align}

The Ritz projection $R_h^{n+1} : \tilde{H} \rightarrow \tilde{H}_h$ with variable coefficient is defined by 
\begin{align}\label{biocnlf-projection-definition2}
	(   \nabla (c - R^{n+1}_h c) , \nabla r_h  ) =0, \quad \forall q_h \in \tilde{H}_h,
\end{align}
and there holds \cite{scott2008,heyinnian2015,ravindran2016,girault2012}
\begin{align}
	\| c -R^{n+1}_h c\|_{L^2} + h \| \nabla(c-R^{n+1}_h c )\| \leq C h^2 \| c\|_{H^2},\label{biocnlf-projection-c}\\
	\| R^{n+1}_h c \|_{L^{\infty}} + \|R^{n+1}_h c \|_{W^{1,4}} \leq C \| c\|_{H^2}.
\end{align}

Denote
\begin{align*}
	&\u^{n+1} -\u^{n+1}_h= \u^{n+1} - \P^{n+1}_h \u^{n+1} + \P^{n+1}_h \u^{n+1}-\u^{n+1}_h = \eta^{n+1}_\u+\e^{n+1}_{\u},\\
	&	p^{n+1} -p^{n+1}_h= p^{n+1} - \rho^{n+1}_h p^{n+1} + \rho^{n+1}_h p^{n+1}-p^{n+1}_h = \eta^{n+1}_p+e^{n+1}_{p},\\
	&	c^{n+1} -c^{n+1}_h= c^{n+1} - R^{n+1}_h c^{n+1} + R^{n+1}_h c^{n+1}-c^{n+1}_h = \eta^{n+1}_c+e^{n+1}_{c}.
\end{align*}

\textbf{Decoupled Crank--Nicolson Leap-Frog Algorithm:}

 \textbf{Step I:} We find the first step iteration $(\u^1_h,p^1_h)  \in \V_h \times M_h$ by 
\begin{align}\label{biobdf-53}
	&(\frac{\u^1_h-\u^0_h}{\tau},\v_h) + \nu(c^0_h+\alpha) (\nabla \u_h^1 ,\nabla \v_h) 	 + B(\u^0_h,\u^{1}_h , \v_h) -(\nabla \cdot \v_h,p^{1}_h) +(\nabla \cdot \u^{1}_h,q_h)\nonumber\\
	=& -g( (1+\gamma c_h^0)\i_2,\v_h )+(\f^{1},\v_h), \quad \forall ~(\v_h,q_h ) \in \V_h \times M_h.
\end{align}

and find $c_h^1 \in \tilde{H}_h$ by
\begin{align}\label{biobdf-13a}
	(\frac{c^1_h-c^0_h}{\tau},\phi_h) + \theta (\nabla c^{1}_h,\nabla \phi_h ) + b(\u^0_h,c^{1}_h, \phi_h ) - U(c^0_h ,\frac{\partial \phi_h}{\partial x_2}) = U\alpha(1,\frac{\partial \phi_h}{\partial x_2}), \quad \forall ~\phi_h \in \tilde{H}_h.
\end{align}

 \textbf{Step II:} 
Given $\u^n_h,c^n_h $, find $(\u^{n+1}_h,p^{n+1}_h)$, with $n=1,2,\dots,N-1,$ such that 

\begin{align}
(&\frac{\u^{n+1}_h-\u^{n-1}_h}{2\tau},\v_h) + \A(c^n_h,\frac{\u^{n+1}_h+\u^{n-1}_h}{2},\v_h)+B(\u^n_h,\frac{\u^{n+1}_h+\u^{n-1}_h}{2},\v_h)  \label{biocnlf-1}\\
&-(p^n_h,\nabla \cdot \v_h)= - \big( g (1+\gamma c^n_h)i_2,\v_h\big) + (\f^n,\v_h), \quad \forall \v_h \in \V_h,\notag\\
&(\nabla \cdot \frac{\u^{n+1}_h+\u^{n-1}_h}{2},q_h) =0, \quad \forall q_h \in M_h.\label{biocnlf-2}
\end{align}

Given $\c^0_h,c^1_h$, find $c^{n+1}_h \in \tilde{H},$ with $n=1,2\dots,N-1$, such that
\begin{align}
	(&\frac{c^{n+1}_h-c^{n-1}_h}{2 \tau}, \phi_h)+\theta (\nabla \frac{c^{n+1}_h+c^{n-1}_h}{2},\phi_h)+b(\u^n_h,\frac{c^{n+1}_h+c^{n-1}_h}{2},\phi_h)\notag \\
	&-U(c^n_h,\frac{\partial \phi_h}{\partial x_2})=U\alpha(1,\frac{\partial \phi_h}{\partial x_2}),\label{biocnlf-3} \quad \forall \phi_h \in \tilde{H}.
\end{align}

\begin{remark}
This semi-implicit treatment not only leads to a linear system at each time step, 
but also plays a crucial role in maintaining the nearly unconditional stability 
and preserving the convergence order of the corresponding nonlinear implicit schemes.

It should be noted that the CNLF scheme is a three-level time discretization method. 
The initial data $(u_h^0, \theta_h^0)$ are typically taken from the exact solution. 
To obtain the first-step approximation $(u_h^1, \theta_h^1)$, an auxiliary time-stepping scheme can be employed. 
It is important to emphasize that the accuracy of this initial step has a significant impact on the overall convergence behavior of the CNLF method. 
For simplicity, in this work, the first-step values $(u_h^1, \theta_h^1)$ are provided by the back Euler method.  The corresponding stability results and convergent analysis at the first time step have been studied in \cite{lichenyangbio,liyuanbio2025}, 
and $(u_h^1, \theta_h^1)$ can be computed by the Crank--Nicolson linear extrapolation (CNLE) scheme.
\end{remark}

Next we will prove the stability of the decoupled CNLF Algorithm by the following theorem.
\begin{theorem}
	(Unconditional stability) $(\u^{n+1}_h, p^{n+1}_h,c^{n+1}_h)$ are the approximation solution of the system (\ref{biocnlf-1})--(\ref{biocnlf-3}), suppose that the conditions of Assumption \textbf{A1}, we have 
\begin{align}\label{biocnlf-stability-result}
	\| \u&^{n+1}_h\|^2_{L^2} + \| c^{n+1}_h\|^2_{L^2} + \kappa \tau \sum_{n=1}^{N} \|\nabla( \u^{n+1}_h+\u^{n-1}_h)\|^2_{L^2}+\theta \tau \sum_{n=1}^{N} \|\nabla( c^{n+1}_h+c^{n-1}_h)\|^2_{L^2} \\
	\leq & C (\|\u^0_h\|^2_{L^2}+\|\u^1_h\|^2_{L^2}+\|c^0_h\|^2_{L^2}+\|c^1_h\|^2_{L^2}) + C \tau \sum_{n=1}^{N} ( \| \f^n\|^2_{L^2} + |\Omega|) \nonumber, \quad \forall \, 1\leq n \leq N-1.
\end{align}
\end{theorem}
\begin{proof}
Taking $\v_h=2 \tau (\u^{n+1}_h+\u^{n-1}_h), q_h=4 \tau  p^n_h,\phi_h= 2\tau (c^{n+1}_h+c^{n-1}_h)$ in $(\ref{biocnlf-1})-(\ref{biocnlf-3})$, we can obtain
\begin{align}
	\big(&\frac{\u^{n+1}_h-\u^{n-1}_h}{2\tau},2 \tau (\u^{n+1}_h+\u^{n-1}_h)\big) + \A\big(c^n_h,\frac{\u^{n+1}_h+\u^{n-1}_h}{2},2 \tau (\u^{n+1}_h+\u^{n-1}_h) \big)\notag\\
	&+B\big(\u^n_h,\frac{\u^{n+1}_h+\u^{n-1}_h}{2},2 \tau (\u^{n+1}_h+\u^{n-1}_h) \big)-\big(p^n_h,2\tau \nabla \cdot (\u^{n+1}_h+\u^{n-1}_h) \big) \\
	=& - \big( g (1+\gamma c^n_h)i_2,2 \tau (\u^{n+1}_h+\u^{n-1}_h) \big) + \big(\f^n,2 \tau (\u^{n+1}_h+\u^{n-1}_h) \big), \quad \forall \v_h \in \V_h,\notag\\
	&(\nabla \cdot \frac{\u^{n+1}_h+\u^{n-1}_h}{2},4 \tau  p^n_h) =0, \quad \forall q_h \in M_h.\\
		(&\frac{c^{n+1}_h-c^{n-1}_h}{2 \tau}, 2\tau (c^{n+1}_h+c^{n-1}_h))+\theta (\nabla \frac{c^{n+1}_h+c^{n-1}_h}{2},2\tau (c^{n+1}_h+c^{n-1}_h))\notag\\
	&	+b(\u^n_h,\frac{c^{n+1}_h+c^{n-1}_h}{2},2\tau (c^{n+1}_h+c^{n-1}_h))-2\tau U(c^n_h,\frac{\partial  (c^{n+1}_h+c^{n-1}_h)}{\partial x_2}) =2\tau U\alpha(1,\frac{\partial  (c^{n+1}_h+c^{n-1}_h)}{\partial x_2}).
\end{align}

Summing up (\ref{biocnlf-1}), (\ref{biocnlf-2}) and (\ref{biocnlf-3}) and noticing that (\ref{biocnlf-assumption1}) and (\ref{biobdf-2}), we have
\begin{align}
	\| \u&^{n+1}_h\|^2_{L^2}+\| c^{n+1}_h\|^2_{L^2} -\| \u^{n-1}_h\|^2_{L^2}-\| c^{n-1}_h\|^2_{L^2}+\kappa \tau \| \nabla(\u^{n+1}_h+\u^{n-1}_h)\|^2_{L^2}+\theta \tau \|\nabla (c^{n+1}_h+c^{n-1}_h)\|^2_{L^2}\\
	=&-2\tau\big(g(1+\gamma c^n_h)i_2,\u^{n+1}_h+\u^{n-1}_h\big)+2\tau (\f^n,\u^{n+1}_h+\u^{n-1}_h)+2\tau U (c^n_h,\frac{\partial c^{n+1}_h+c^{n-1}_h }{\partial x_2})\notag
	\\&+2\tau U \alpha( 1,\frac{\partial c^{n+1}_h+c^{n-1}_h}{\partial x_2})\notag.
\end{align}

To estimate the right-hand side of the above equation, we apply the Cauchy--Schwarz and Young inequalities to each term. This yields the following upper bound, which is essential for the subsequent stability analysis.
\begin{align}
	| -2\tau\big(g(1+\gamma c^n_h)i_2,\u^{n+1}_h+\u^{n-1}_h\big) | &\leq C\|1+c^n_h\|_{L^2}\| \nabla ( \u^{n+1}_h+\u^{n-1}_h)\|_{L^2}\notag\\
	&\leq  \epsilon_1 \tau\| \nabla (\u^{n+1}_h+\u^{n-1}_h)\|^2_{L^2} +C\tau \|c^n_h\|^2_{L^2}+ C\tau |\Omega|,\\
	|  2\tau (\f^n,\u^{n+1}_h+\u^{n-1}_h)| &\leq \epsilon_1 \tau \| \nabla (\u^{n+1}_h+\u^{n-1}_h)\|^2_{L^2} +C\tau \|\f^n\|^2_{L^2},\\
	|2\tau U (c^n_h,\frac{\partial c^{n+1}_h+c^{n-1}_h }{\partial x_2})| &\leq \epsilon_2 \tau \|\nabla (c^{n+1}_h+c^{n-1}_h)\|^2_{L^2}+C \tau \|c^n_h\|^2_{L^2}\\
	|2\tau U \alpha( 1,\frac{\partial c^{n+1}_h+c^{n-1}_h}{\partial x_2})| &\leq \epsilon_2 \tau \|\nabla (c^{n+1}_h+c^{n-1}_h)\|^2_{L^2}+C\tau |\Omega|.
\end{align}

Combining with the above inequalities, summing up from $n=2$ to $N$, applying the discrete Gronwall inequality in Lemma \ref{biobdf-11}, for sufficiently small $\epsilon_1,\epsilon_2$, we can obtain the discrete energy inequalities.
\end{proof}
\section{Convergent analysis}\label{sec-4}
In this section, we concentrate on an error analysis for the presented method for the Bioconvection. Optimal error estimates are obtained. Now, we will present the main theorem in this paper, and the complete analysis will be show in the following.
\begin{theorem}{\label{biocnlf-theorem2}}
	(Convergence of CNLF). Consider the CNLF algorithm (\ref{biocnlf-1})--(\ref{biocnlf-3}). In terms of  the assumption \textbf{A1}  and \textbf{A2}, for $0\leq i \leq N$, Under the time step size restriction $C\tau \leq h$,  there is a positive constant $C$ independent of the mesh size and the time step size such that
	\begin{align}
\max_{0 \leq i \leq N} (\| \mathbf{u}^i - \mathbf{u}_h^i \|_{L^2}^2 + \| c^i - c_h^i \|_{L^2}^2) \leq C (\tau^4 + h^4).
	\end{align}
\end{theorem}

In order to derive the error estimate, we need the following inequalities
\begin{lemma}
	Under the assumption \textbf{A2}, the following inequalities satisfies:
	\begin{align*}
		\tau \sum_{n=1}^{N-1} 
		\left\| \u_t(t_n) - \frac{\u^{n+1} - \u^{n-1}}{2} \right\|_{L^2}^2
		&\le C
		\tau^4,\\[1em]
		\tau \sum_{n=1}^{N-1} 
		\left\| c_t(t_n) - \frac{c^{n+1} - c^{n-1}}{2} \right\|_{L^2}^2
		&\le C
		\tau^4,\\[1em]
		\tau \sum_{n=1}^{N-1} 
		\left\| \nabla \left( \u^n - \frac{\u^{n+1} + \u^{n-1}}{2} \right) \right\|_{L^2}^2
		&\le C
		\tau^4,\\[1em]
		\tau \sum_{n=1}^{N-1} 
		\left\| \nabla \left( c^n - \frac{c^{n+1} + c^{n-1}}{2} \right) \right\|_{L^2}^2
		&\le C
		\tau^4.
	\end{align*}
\end{lemma}

\begin{proof}
	The proof of these inequalities is similiar to \cite{kubacki2013}. By integrating by parts twice and the Cauchy--Schwarz inequality, we have 
	\begin{align*}
		&\sum_{n=1}^{N-1} \left\| \frac{\u_t(t_n) - \u^{n+1} - \u^{n-1}}{2\tau} \right\|_{L^2} \\
		=& \frac{1}{4\tau^2} \int_{\Omega} \sum_{n=1}^{N-1} 
		\left( \int_{t_n}^{t^{n+1}} (t - t^{n+1}) \u_{tt}\, dt 
		+ \int_{t_{n-1}}^{t_n} (t - t^{n-1}) \u_{tt}\, dt \right)^2 dx \\
		=& \frac{1}{4 \tau^2} \int_{\Omega} 
		\sum_{n=1}^{N-1} \left( \int_{t_n}^{t^{n+1}} 
		\frac{(t - t^{n+1})^2}{2} \u_{ttt}\, dt 
		+ \int_{t_{n-1}}^{t_n} \frac{(t - t^{n-1})^2}{2} \u_{ttt}\, dt \right)^2 dx \\
		\le& \frac{1}{2 \tau^2} \int_{\Omega} 
		\sum_{n=1}^{N-1} \frac{ \tau^5}{20} 
		\left( \int_{t_{n-1}}^{t^{n+1}} |\u_{ttt}|^2\, dt \right) dx \\
		\le &\frac{\tau^3}{20} \int_{\Omega} \int_0^T |\u_{ttt}|^2\, dt\, dx\\
		\le& C\tau^3.
	\end{align*}
	
	similarily
	\begin{align*}
		&\sum_{n=1}^{N-1} \left\| \nabla \left( \u^n - \frac{\u^{n+1} + \u^{n-1}}{2} \right) \right\|^2_{L^2} \\
		=& \int_{\Omega} \sum_{n=1}^{N-1} \left| \nabla \left( \frac{\u^n - \u^{n+1}}{2} + \frac{\u^n - \u^{n-1}}{2} \right) \right|^2 dx \\
		= &\frac{1}{4} \int_{\Omega} \sum_{n=1}^{N-1} \left| \left( \nabla \u^n - \nabla \u^{n+1} \right) + \left( \nabla \u^n - \nabla \u^{n-1} \right) \right|^2 dx \\
		=& \frac{1}{4} \int_{\Omega} \sum_{n=1}^{N-1} \left| \int_{t_{n-1}}^{t^n} \nabla \u_t \, dt + \int_{t_n}^{t^n} \nabla \u_t \, dt \right|^2 dx \\
		=& \frac{1}{4} \int_{\Omega} \sum_{n=1}^{N-1} \left| \int_{t_{n-1}}^{t^n} (t - t^n)' \nabla \u_t \, dt + \int_{t_{n-1}}^{t^n} (t - t^k)' \nabla \u_t \, dt \right|^2 dx \\
		=& \frac{1}{4} \int_{\Omega} \sum_{n=1}^{N-1} \left| -\tau \int_{t_{n-1}}^{t^{n+1}} \nabla \u_t \, dt + \int_{t_n}^{t^{n+1}} (t - t^n) \nabla \u_t \, dt + \int_{t_{n-1}}^{t^n} (t^n - t) \nabla \u_t \, dt \right|^2 dx \\
		\leq& \frac{1}{2} \int_{\Omega} \sum_{n=1}^{N-1} \left( \tau^2 \left| \int_{t_{n-1}}^{t^{n+1}} \nabla \u_t \, dt \right|^2 + \left| \int_{t_n}^{t^{n+1}} (t - t^n) \nabla \u_t \, dt \right|^2 + \left| \int_{t_{n-1}}^{t^n} (t^n - t) \nabla \u_t \, dt \right|^2 \right) dx\\
		\leq& \frac{1}{2} \int_{\Omega} \sum_{n=1}^{N-1} \left( \tau^3 \int_{t_{n-1}}^{t^{n+1}} |\nabla \u_{tt}|^2 \, dt + \frac{\tau^3}{3} \int_{t_{n-1}}^{t^{n+1}} |\nabla \u_{tt}|^2 \, dt \right) dx \\
		\leq&C \tau^3.
	\end{align*}
	
\end{proof}

\begin{theorem}\label{biocnlf-theorem1}
	Based on the Assumption A1 and A2, let $(\u^{i}, p^{i},c^{i})$  and $(\u_h^{i}, p_h^{i},c_h^{i})$ be the solutions of the continuous model (\ref{biobdf-4}) and the finite element discrete scheme (\ref{biocnlf-1}) -- (\ref{biocnlf-3}) , respectively. Under the time step size restriction $C\tau \leq h$, there exists some $C_1>0$ independent of $\tau$ and $h$ such that $ \forall~0 \leq m \leq N-1$
	\begin{align}\label{biobdf-55}
		\| \e^{m+1}_{\u} \|^2_{L^2} +\| e^{m+1}_{c} \|^2_{L^2} +\tau k^{-1} \sum_{i=0}^{m}  \| \nabla \e^{i+1}_{\u} \|^2_{L^2} +\tau \theta  \sum_{i=0}^{m} \| \nabla e^{i+1}_{c} \|^2_{L^2} \leq C_1 (\tau^4+h^4).
	\end{align}
\end{theorem}

\begin{proof}
	The main idea to proof Theorem \ref{biocnlf-theorem2} is the mathematical induction, so we need to assume that (\ref{biobdf-55}) is valid for $ i = n $, i.e.
	\begin{align}
	\| e_{\mathbf{u}}^n \|_{L^2}^2 + \| e_c^n \|_{L^2}^2 \leq C(\tau^2 + h^4).
	\end{align}

By taking $C\tau \leq h$, we can see that
	\begin{align}
	\| e_{\mathbf{u}}^n \|_{L^2}^2 + \| e_c^n \|_{L^2}^2 \leq C h^4.
\end{align}
	We will prove (\ref{biobdf-55}) is valid for \( i = n + 1 \). First, we can derive the following variational formulation at the time $t=t_n$ from the continue model (\ref{biobdf-4}).
\begin{align}
		(&\frac{\partial \u (t_n)}{\partial t} , \v_h) + A(c^n,\u^n,\v_h) + B(\u^n,\u^n,\v_h) - ( \div \v,p^n_h) + (\div \u^n,q_h) \label{biocnlf-continue1}\\
	&=(\f^n,\v_h)- g((1+\gamma c^n)\i_2,\v_h),\quad \forall \, (\v_h, q_h) \in \V_h \times M_h,\notag\\
		(&\frac{\partial c(t_n)}{\partial t},\phi_h) +b(\u^n,c^n,\phi_h) -U(c^n,\frac{\partial \phi_h}{\partial x_2}) = U\alpha (1, \frac{\partial \phi_h}{\partial x_2}), \quad \forall \, \phi_h\in \tilde{H}.\label{biocnlf-continue2}
\end{align}

Substracting (\ref{biocnlf-1})--(\ref{biocnlf-3}) from (\ref{biocnlf-continue1})--(\ref{biocnlf-continue2}), and noticing that projection operator (\ref{biocnlf-projection1}), (\ref{biocnlf-projection-definition2}) and the divergence-free space (\ref{divergence-free}), we have the following error equation

\begin{align}\label{biocnlf-errorequation1}
	&\left( \frac{\e_{\u}^{n+1} - \e_{\u}^{n-1}}{2\tau},\v_h \right) + \left( \frac{e_c^{n+1} - e_c^{n-1}}{2\tau}, \phi_h \right) \notag\\
	&+A(c^n,\frac{\e^{n+1}_\u+\e^{n-1}_\u}{2},\v_h) + \theta \left(  \nabla \frac{e^{n+1}_c + e^{n-1}_c}{2}, \nabla \phi_h \right) \notag\\
	&\leq \left( \frac{\P^{n+1}_h \u^{n+1} - \P^{n+1}_h \u^{n-1}}{2\tau} - \u_t(t_n), \v_h \right) + \left( \frac{ R^{n+1}_h c^{n+1}-R^{n+1}_h c^{n-1}}{2\tau} - c_t(t_n), \phi_h
		 \right) \\
		 &-A(c^n,\u^n-\frac{\u^{n+1}+\u^{n-1}}{2},\v_h)\notag+A\big((c^n-c^n_h),\frac{\e_\u^{n+1}+\e^{n-1}_\u}{2},\v_h\big)\\
	&-A\big((c^n-c^n_h,\frac{\P^{n+1}\u^{n+1} +\P^{n+1}\u^{n-1} }{2},\v_h)\big)
	+ \theta \left(\nabla( \frac{R^{n+1}_h c^{n+1}+R^{n+1}_h c^{n-1}}{2} - c^n), \nabla \phi_h \right) \notag\\
	&+ \left| B \left( \u^n, \u^n, \v_h \right) - B \left( \u_h^n, \frac{\u_h^{n+1} + \u_h^{n-1}}{2}, \v_h \right) \right| + \left| b \left( \u^n, c^n, \phi_h \right) - b \left( \u_h^n, \frac{c_h^{n+1} + c_h^{n-1}}{2}, \phi_h \right) \right| \notag\\
	&+\left| \big( g (1+\gamma c^n)i_2,\v_h\big) -\big( g (1+\gamma c^n_h)i_2,\v_h\big)  \right|+\left| U(c^n,\frac{\partial \phi_h}{\partial x_2}) -U(c^n_h,\frac{\partial \phi_h}{\partial x_2})\right|.\notag
\end{align}
Setting $(\v_h,\phi_h)=\big(2\tau(\e^{n+1}_\u+\e^{n-1}_\u) ,2\tau(e^{n+1}_c+e^{n-1}_c)\big)$ in (\ref{biocnlf-errorequation1}) and noticing that (\ref{biocnlf-assumption1}), we can get
\begin{align}\label{biocnlf-errorequation2}
	\| \e&^{n+1}_\u\|^2_{L^2}+\| e^{n+1}_c\|^2_{L^2}-\| \e^{n-1}_\u\|^2_{L^2}-\| e^{n-1}_c\|^2_{L^2}\\
	&+ \kappa \tau \| \nabla (\e^{n+1}_\u+\e^{n+1}_\u)\|^2_{L^2}+ \theta\tau \|   \nabla (e^{n+1}_c + e^{n-1} _c)\|^2_{L^2} \notag\\
	\leq &2\tau \left( \frac{\P^{n+1}_h \u^{n+1} - \P^{n+1}_h \u^{n-1}}{2\tau} - \u_t(t_n), \e^{n+1}_\u+\e^{n-1}_\u \right)\notag\\
	&+ 2\tau \left( \frac{ R^{n+1}_h c^{n+1}-R^{n+1}_h c^{n-1}}{2\tau} - c_t(t_n), e^{n+1}_c+e^{n-1}_c
	\right) \notag\\
	&+2\tau \big| A(c^n,\u^n-\frac{\u^{n+1}+\u^{n-1}}{2},\e^{n+1}_\u+\e^{n-1}_\u)\notag\big|\notag\\
	&+2\tau \big| A\big((c^n-c^n_h),\frac{\e_\u^{n+1}+\e^{n-1}_\u}{2},\e^{n+1}_\u+\e^{n-1}_\u\big)\big|\notag\\
	&+2\tau \big| A\big((c^n-c^n_h,\frac{\P^{n+1}\u^{n+1} +\P^{n+1}\u^{n-1} }{2},\e^{n+1}_\u+\e^{n-1}_\u)\big)\big|
\notag\\
	&+2 \theta \tau \left( \nabla( \frac{R^{n+1}_h c^{n+1}+R^{n+1}_h c^{n-1}}{2} - c^n), \nabla (e^{n+1}_c+e^{n-1}_c) \right) \notag\\
	&+2\tau \left|  B \left( \u^n, \u^n, \e^{n+1}_\u+\e^{n-1}_\u \right) -  B \left( \u_h^n, \frac{\u_h^{n+1} + \u_h^{n-1}}{2}, \e^{n+1}_\u+\e^{n-1}_\u \right) \right| \notag\\
	&+2\tau  \left| b \left( \u^n, c^n, e^{n+1}_c+e^{n-1}_c \right) - b \left( \u_h^n, \frac{c_h^{n+1} + c_h^{n-1}}{2}, e^{n+1}_c+e^{n-1}_c \right) \right| \notag\\
	&+2\tau \left| \big( g (1+\gamma c^n)i_2,\e^{n+1}_\u+\e^{n-1}_\u\big) -\big( g (1+\gamma c^n_h)i_2,\e^{n+1}_\u+\e^{n-1}_\u \big)  \right|\notag\\
	&+2\tau\left| U(c^n,\frac{\partial e^{n+1}_c+e^{n-1}_c }{\partial x_2}) -U(c^n_h,\frac{\partial e^{n+1}_c+e^{n-1}_c }{\partial x_2})\right|\notag\\
	=&\sum_{i=1}^{10}X_i.\notag
\end{align}

Next, we need to estimate every term on the right hand side.
By the H\"{o}lder inequality, Young inequality, we can have 
\begin{align}
	X_1=
&2\tau \left( \frac{\P^{n+1}_h \u^{n+1} - \P^{n+1}_h \u^{n-1}}{2\tau} - \u_t(t_n), \e^{n+1}_\u+\e^{n-1}_\u \right)\notag \\
=& 2\tau \left. \left( \frac{\u^{n+1} - \u^{n-1}}{2\tau} - \u_t(t_n), \e^{n+1}_\u + \e^{n-1}_\u \right) -2\tau \left( \frac{\eta_\u^{n+1} - \eta_\u^{n-1}}{2\tau}, \e^{n+1}_\u + \e^{n-1}_\u \right) \right. \\
	\leq & \epsilon_3 \tau \left\|\nabla  (\e^{n+1}_\u + \e^{n-1}_\u )\right\|_{L^2}^2 + C\tau \left\| \frac{\u^{n+1} - \u^{n-1}}{2\tau} - \u_t(t_n) \right\|_0^2 + C\tau \|\frac{\eta_\u^{n+1} - \eta_\u^{n-1}}{2\tau}\|^2_{L^2}\notag \\
	\leq&\epsilon_3 \tau \left\|\nabla  (\e^{n+1}_\u + \e^{n-1}_\u )\right\|_{L^2}^2 + C \tau\left\| \frac{\u^{n+1} - \u^{n-1}}{2\tau} - \u_t(t_n) \right\|_{L^2}^2  + C\int_{\Omega} \int_{t^{n-1}}^{t^{n+1}} |\eta_{\u t}|^2 dt dx \notag.
\end{align}

Similarly
\begin{align}
	X_2=
	&2\tau \left( \frac{R^{n+1}_h c^{n+1} - R^{n+1}_h c^{n-1}}{2\tau} - c_t(t_n), e^{n+1}_c+e^{n-1}_c \right)\notag \\
	=& 2\tau \left. \left( \frac{c^{n+1} - c^{n-1}}{2\tau} - c_t(t_n), e^{n+1}_c + e^{n-1}_c \right) -2\tau \left( \frac{\eta_c^{n+1} - \eta_c^{n-1}}{2\tau}, e^{n+1}_c + e^{n-1}_c \right) \right. \\
	\leq & \epsilon_3 \tau \left\|\nabla  (e^{n+1}_c + e^{n-1}_c )\right\|_{L^2}^2 + C\tau \left\| \frac{c^{n+1} - c^{n-1}}{2\tau} - c_t(t_n) \right\|_0^2 + C\tau \|\frac{\eta_c^{n+1} - \eta_c^{n-1}}{2\tau}\|^2_{L^2}\notag \\
	\leq&\epsilon_3 \tau \left\|\nabla  (e^{n+1}_c + e^{n-1}_c )\right\|_{L^2}^2 + C \tau\left\| \frac{c^{n+1} - c^{n-1}}{2\tau} - c_t(t_n) \right\|_{L^2}^2  + C\int_{\Omega} \int_{t^{n-1}}^{t^{n+1}} |\eta_{c t}|^2 dt dx \notag.
\end{align}

In terms of (\ref{biocnlf-h2w100}) and assumption \textbf{A2},  one has 
\begin{align}
	X_3 \leq& C \tau \| c^n\|_{L^{\infty}}  \| \nabla (\u^n - \frac{\u^{n+1} + \u^{n-1})}{2}\|_{L^2}  \| \nabla(\e^{n+1}_\u+\e^{n-1}_\u)\|_{L^2}\\
	\leq & C \tau  \| \nabla (\u^n - \frac{\u^{n+1} + \u^{n-1})}{2}\|^2_{L^2}+ \epsilon_3 \tau \| \nabla(\e^{n+1}_\u+\e^{n-1}_\u)\|^2_{L^2}.\notag
\end{align}

By the mathematical induction and inverse inequalities (\ref{biocnlf-inverse2}), for sufficiently small $h$ such that $Ch\leq \frac{\kappa}{2}$, we get
\begin{align}
X_4\leq&C\tau \| \eta^n_c+e^n_c\|_{L^{\infty}}\| \nabla(\e^{n+1}_\u+\e^{n-1}_\u)\|_{L^2}\| \nabla(\e^{n+1}_\u+\e^{n-1}_\u)\|_{L^2}\\
\leq &C\tau h^{-1} \| \eta^n_c+e^n_c\|_{L^{2}}\|\nabla(\e^{n+1}_\u+\e^{n-1}_\u)\|_{L^2}^2\notag\\
\leq &
C\tau h \|\nabla(\e^{n+1}_\u+\e^{n-1}_\u)\|_{L^2}^2\notag\\
\leq& \frac{\kappa}{2}\tau\|\nabla(\e^{n+1}_\u+\e^{n-1}_\u)\|_{L^2}^2.\notag
\end{align}

According to (\ref{biocnlf-w24}), one has
\begin{align}
	X_5\leq &C \tau \| \eta^n_c+e^n_c\|_{L^{2}}  \| \nabla \frac{\P^{n+1}\u^{n+1} +\P^{n+1}\u^{n-1} }{2}\|_{L^\infty}\|\nabla(\e^{n+1}_\u+\e^{n-1}_\u)\|_{L^2}\\
	\leq& C \tau\| \eta^n_c+e^n_c\|_{L^{2}} ( \| \u^{n+1}\|_{W^{2,4}}+\| \u^{n-1}\|_{W^{2,4}}) |\nabla(\e^{n+1}_\u+\e^{n-1}_\u)\|_{L^2}\notag\\
	\leq& C\tau h^4 + C \tau \|e^n_c\|^2_{L^2}+  \epsilon_3 \tau \| \nabla(\e^{n+1}_\u+\e^{n-1}_\u)\|^2_{L^2}.\notag
\end{align}

By using the projection definition (\ref{biocnlf-projection-definition2}), the H\"{o}lder inequality and Young inequality, on has
\begin{align}
	X_6=
&2 \theta \tau \left( \nabla( \frac{R^{n+1}_h c^{n+1}+R^{n+1}_h c^{n-1}}{2} - c^n), \nabla (e^{n+1}_c+e^{n-1}_c) \right)\notag\\
=&2 \theta \tau \left( \nabla( \frac{ \eta_c^{n+1}+ \eta_c^{n-1}}{2}), \nabla (e^{n+1}_c+e^{n-1}_c) \right)+2 \theta \tau \left( \nabla( \frac{ c^{n+1}+ c^{n-1}}{2} -c^n), \nabla (e^{n+1}_c+e^{n-1}_c) \right)\notag\\
\leq & \epsilon_4\tau \| \nabla (e^{n+1}_c+e^{n-1}_c)\|^2_{L^2}+ C \tau\| \nabla ( \frac{ c^{n+1}+ c^{n-1}}{2} -c^n) \|^2_{L^2}.
\end{align}

By recombination, we can get
\begin{align}\label{biocnlf-Btrin}
	X_7=
&2\tau \left|  B \left( \u^n, \u^n, \e^{n+1}_\u+\e^{n-1}_\u \right) -  B \left( \u_h^n, \frac{\u_h^{n+1} + \u_h^{n-1}}{2}, \e^{n+1}_\u+\e^{n-1}_\u \right) \right|  \\
	= &2\tau B\left( \u^n - \frac{\u^{n+1} + \u^{n-1}}{2}, \u^n, \e^{n+1}_\u+\e^{n-1}_\u \right) +2\tau B \left( \frac{\eta_\u^{n+1}+\eta_\u^{n-1}}{2}, \u^n, \e^{n+1}_\u + \e^{n-1}_\u \right) \notag\\
	& + 2\tau B \left( \frac{ \P^{n+1}_h \u^{n+1}+\P^{n+1}_h \u^{n-1}}{2}, \eta^n_\u, \e^{n+1}_\u + \e^{n-1}_\u \right)\notag \\
	&+ 2\tau B \left( \frac{\P^{n+1}_h \u^{n+1}+\P^{n+1}_h \u^{n-1}}{2}, \P^{n+1}_h \u^n -  \frac{ \P^{n+1}_h \u^{n+1}+\P^{n+1}_h \u^{n-1}}{2}, \e^{n+1}_\u + \e^{n-1}_\u \right) \notag\\
	& + 2\tau B\left( \frac{\P^{n+1}_h \u^{n+1}+\P^{n+1}_h \u^{n-1}}{2} - \P^{n+1}_h \u^n, \frac{\P^{n+1}_h \u^{n+1}+\P^{n+1}_h \u^{n-1}}{2}, \e^{n+1}_\u + \e^{n-1}_\u  \right)\notag \\&+ 2\tau B \left( \e^n_\u, \frac{\u^{n+1} + \u^{n-1}}{2}, \e^{n+1}_\u + \e^{n-1}_\u\right) - 2\tau B \left( \e^n_\u, \frac{\eta^{n+1}_\u + \eta_\u^{n-1}}{2}, \e^{n+1}_\u + \e^{n-1}_\u\right)\notag \\
	&-2\tau B \left( \e^n_\u, \frac{\e^{n+1}_\u + \e^{n-1}_\u}{2}, \e^{n+1}_\u + \e^{n-1}_\u\right)+2\tau B \left( \P^{n+1}_h\u^n, \frac{\e^{n+1}_\u + \e^{n-1}_\u}{2}, \e^{n+1}_\u + \e^{n-1}_\u\right)\notag\\
	=&\sum_{i=1}^{9}Y_i.\notag
\end{align}

By the H\"{o}lder inequality, Young inequality, we can estimate the right hand side of (\ref{biocnlf-Btrin}),
\begin{align}
	|Y_1| \leq&C \tau  \|\u^n - \frac{\u^{n+1} + \u^{n-1}}{2}\|_{L^3}  \| \nabla \u^n\|_{L^2} \| \e^{n+1}_\u+\e^{n-1}_\u\|_{L^6}\\
	\leq & C \tau \| \nabla (\u^n - \frac{\u^{n+1} + \u^{n-1})}{2}\|^2_{L^2}+ \epsilon_3 \tau \| \nabla(\e^{n+1}_\u+\e^{n-1}_\u)\|^2_{L^2},\notag\\
		|Y_2| \leq& C\tau \| \frac{\eta_\u^{n+1}+\eta_\u^{n-1}}{2}\|_{L^2} \| \nabla \u^n\|_{L^{\infty}}\|\e^{n+1}_\u + \e^{n-1}_\u \|_{L^2}\\
		\leq & C\tau h^4 +\epsilon_3 \tau \| \nabla(\e^{n+1}_\u+\e^{n-1}_\u)\|^2_{L^2}.\notag
 \end{align}
 
In terms of (\ref{biobdf-17}) and assumption \textbf{A2}, we have
\begin{align}
	|Y_3| =& \left| -2\tau B \left( \frac{ \P^{n+1}_h \u^{n+1}+\P^{n+1}_h \u^{n-1}}{2},  \e^{n+1}_\u + \e^{n-1}_\u ,\eta^n_\u\right) \right|\\
	\leq& C\|\frac{ \P^{n+1}_h \u^{n+1}+\P^{n+1}_h \u^{n-1}}{2}\|_{L^\infty} \| \e^{n+1}_\u + \e^{n-1}_\u \|_{L^2} \|\eta^n_\u\|_{L^2}\notag\\
	\leq & C \tau \| \u^{n+1} + \u^{n-1}\|_{L^\infty}   \| \nabla(\e^{n+1}_\u+\e^{n-1}_\u) \| _{L^2}\|\eta^n_\u\|_{L^2}\notag\\
	\leq & C\tau h^4 +\epsilon_3 \tau \| \nabla(\e^{n+1}_\u+\e^{n-1}_\u)\|^2_{L^2}.\notag
\end{align}

By the H\"{o}lder inequality, Young inequality, one has 
\begin{align}
	&|Y_4|+|Y_5| \\
		\leq &C \tau \| \frac{\P^{n+1}_h \u^{n+1}+\P^{n+1}_h \u^{n-1}}{2}\|_{L^3} \| \nabla (\P^{n+1}_h \u^n -  \frac{ \P^{n+1}_h \u^{n+1}+\P^{n+1}_h \u^{n-1}}{2})\|_{L^2} \| \e^{n+1}_\u + \e^{n-1}_\u \|_{L^6}\notag\\
		\leq& C \tau \| \u^{n+1} +\u^{n-1}\|_{L^3}\| \nabla (\u^n -  \frac{\u^{n+1}+ \u^{n-1}}{2})\|_{L^2} \| \nabla(\e^{n+1}_\u+\e^{n-1}_\u) \| _{L^2}\notag\\
		\leq &C\tau | \nabla (\u^n -  \frac{\u^{n+1}+ \u^{n-1}}{2})\|^2_{L^2}+\epsilon_3 \tau \| \nabla(\e^{n+1}_\u+\e^{n-1}_\u)\|^2_{L^2}.\notag
\end{align}

Furthermore

\begin{align}
|Y_6|\leq &2\tau B \left( \e^n_\u, \frac{\u^{n+1} + \u^{n-1}}{2}, \e^{n+1}_\u + \e^{n-1}_\u\right) \\
\leq & C\tau \| \e^n_\u\|_{L^2}  \| \frac{\u^{n+1} + \u^{n-1}}{2} \|_{L^3} \|\e^{n+1}_\u + \e^{n-1}_\u\|_{L^6} \notag\\
\leq & C\tau \|\e^n_\u\|_{L^2} +\epsilon_3 \tau \| \nabla(\e^{n+1}_\u+\e^{n-1}_\u)\|^2_{L^2}.\notag
\end{align}

By using inverse inequality (\ref{biocnlf-inverse1}) and projection error (\ref{biocnlf-projection-u}), for sufficiently small $h$ such that $Ch\leq C_1$, we can derive 
\begin{align}
|Y_7| =&\left | -2\tau B \left( \e^n_\u,\e^{n+1}_\u + \e^{n-1}_\u, \frac{\eta^{n+1}_\u + \eta_\u^{n-1}}{2} \right)\right |\\
\leq &C \tau \|  \e^n_\u\|_{L^2}  \| \nabla(e^{n+1}_\u + \e^{n-1}_\u) \|_{L^2}\| \frac{\eta^{n+1}_\u + \eta_\u^{n-1}}{2}\|_{L^{\infty}}\notag \\
\leq & C \tau h^{-1} \|  \e^n_\u\|_{L^2}  \| \nabla(e^{n+1}_\u + \e^{n-1}_\u) \|_{L^2}\| \frac{\eta^{n+1}_\u + \eta_\u^{n-1}}{2}\|_{L^{2}}\notag\\
\leq &C \tau h\|  \e^n_\u\|_{L^2}  \| \nabla(e^{n+1}_\u + \e^{n-1}_\u) \|_{L^2}\notag\\
\leq & C_1 \tau \|  \e^n_\u\|_{L^2}^2+\epsilon_3 \tau \| \nabla(\e^{n+1}_\u+\e^{n-1}_\u)\|^2_{L^2}.\notag
\end{align}

By the mathematical  induction, and for sufficiently small $h$ such that $Ch\leq C_1$, we have
\begin{align}
	|Y_8| \leq &C\tau \| \e_\u^n\|_{L^\infty} \| \nabla ( \frac{\e^{n+1}_\u + \e^{n-1}_\u}{2})\|_{L^2}\| \e^{n+1}_\u + \e^{n-1}_\u\|_{L^2}\\
	\leq &C \tau h^{-1}\| \e_\u^n\|_{L^2} \| \nabla ( \frac{\e^{n+1}_\u + \e^{n-1}_\u}{2})\|_{L^2}\| \e^{n+1}_\u + \e^{n-1}_\u\|_{L^2}\notag\\
	\leq& C \tau h \| \nabla ( \frac{\e^{n+1}_\u + \e^{n-1}_\u}{2})\|_{L^2}\| \e^{n+1}_\u + \e^{n-1}_\u\|_{L^2}\notag\\
	\leq& C _1 \tau (\| \e^{n+1}_\u\|^2_{L^2}+\| \e^{n-1}_\u\|^2_{L^2}) +\epsilon_3 \tau \| \nabla(\e^{n+1}_\u+\e^{n-1}_\u)\|^2_{L^2}.\notag
\end{align}

For last term, we have 
\begin{align}\label{biocnlf-Btrin2}
		|Y_9| \leq &\| \P^{n+1}_h\u^n\|_{L^{\infty}}\| \nabla \frac{\e^{n+1}_\u + \e^{n-1}_\u}{2}\|_{L^2} \|\e^{n+1}_\u + \e^{n-1}_\u\|_{L^2}\\
		\leq & C _1 \tau (\| \e^{n+1}_\u\|^2_{L^2}+\| \e^{n-1}_\u\|^2_{L^2}) +\epsilon_3 \tau \| \nabla(\e^{n+1}_\u+\e^{n-1}_\u)\|^2_{L^2}.\notag
\end{align}

Substituting (\ref{biocnlf-Btrin})--(\ref{biocnlf-Btrin2}) into $X_7$, one has 
\begin{align}
	|X_7| \leq &C \tau \| \nabla (\u^n - \frac{\u^{n+1} + \u^{n-1})}{2}\|^2_{L^2}+ \epsilon_3 \tau \| \nabla(\e^{n+1}_\u+\e^{n-1}_\u)\|^2_{L^2}\notag
	\\
	&+ C\tau h^4+C \tau (\| \e^{n+1}_\u\|^2_{L^2}+\| \e^{n-1}_\u\|^2_{L^2}+ \|\e^n_\u\|_{L^2}) .
\end{align}

According to the same technique in (\ref{biocnlf-Btrin})--(\ref{biocnlf-Btrin2}), we can handle $X_8$
\begin{align}\label{biocnlf-btrin}
	X_8=
	&2\tau \left|  b \left( \u^n, c^n, e^{n+1}_c+e^{n-1}_c \right) -  B \left( \u_h^n, \frac{c_h^{n+1} + c_h^{n-1}}{2}, e^{n+1}_c+e^{n-1}_c \right) \right|  \\
	= &2\tau B\left( \u^n - \frac{\u^{n+1} + \u^{n-1}}{2}, c^n, e^{n+1}_c+e^{n-1}_c \right) +2\tau B \left( \frac{\eta_\u^{n+1}+\eta_\u^{n-1}}{2}, c^n, e^{n+1}_c +e^{n-1}_c \right) \notag\\
	& + 2\tau B \left( \frac{ \P^{n+1}_h \u^{n+1}+\P^{n+1}_h \u^{n-1}}{2}, \eta^n_c, e^{n+1}_c + e^{n-1}_c \right)\notag \\
	&+ 2\tau B \left( \frac{\P^{n+1}_h \u^{n+1}+\P^{n+1}_h \u^{n-1}}{2}, R^{n+1}_h c^n -  \frac{ R^{n+1}_h c^{n+1}+R^{n+1}_h c^{n-1}}{2}, e^{n+1}_c + e^{n-1}_c \right) \notag\\
	& + 2\tau B\left( \frac{\P^{n+1}_h \u^{n+1}+\P^{n+1}_h \u^{n-1}}{2} - \P^{n+1}_h \u^n, \frac{R^{n+1}_h c^{n+1}+R^{n+1}_h c^{n-1}}{2}, e^{n+1}_c + e^{n-1}_c  \right)\notag \\&+ 2\tau B \left( \e^n_\u, \frac{c^{n+1} + c^{n-1}}{2}, e^{n+1}_c + e^{n-1}_c \right) - 2\tau B \left( \e^n_\u, \frac{\eta^{n+1}_c + \eta_c^{n-1}}{2}, e^{n+1}_c + e^{n-1}_c\right)\notag \\
	&-2\tau B \left( \e^n_\u, \frac{e^{n+1}_c + e^{n-1}_c}{2}, e^{n+1}_c + e^{n-1}_c\right)+2\tau B \left( \P^{n+1}_h\u^n, \frac{e^{n+1}_c + e^{n-1}_c}{2}, e^{n+1}_c + e^{n-1}_c\right)\notag\\
	\leq& C \tau \| \nabla (c^n - \frac{c^{n+1} + c^{n-1})}{2}\|^2_{L^2}+ \epsilon_4 \tau \| \nabla(e^{n+1}_c+e^{n-1}_c)\|^2_{L^2}\notag
	\\
	&+ C\tau h^4+C \tau (\| e^{n+1}_c\|^2_{L^2}+\| e^{n-1}_c\|^2_{L^2}+ \|\e^n_\u\|_{L^2}).\notag
\end{align}

For the last two terms of (\ref{biocnlf-errorequation2}), applying  the H\"{o}lder inequality, Young inequality, we have
\begin{align}
	X_9+X_{10} \leq&2\tau \left| \big( g (1+\gamma c^n)i_2,\e^{n+1}_\u+\e^{n-1}_\u\big) -\big( g (1+\gamma c^n_h)i_2,\e^{n+1}_\u+\e^{n-1}_\u \big)  \right|\\
	&+2\tau\left| U(c^n,\frac{\partial e^{n+1}_c+e^{n-1}_c }{\partial x_2}) -U(c^n_h,\frac{\partial e^{n+1}_c+e^{n-1}_c }{\partial x_2})\right|\notag\\
	\leq &C \tau \| \eta^n_c + e^n_c\|_{L^2} \| \e^{n+1}_\u+\e^{n-1}_\u \|_{L^2}+C \tau \| \eta^n_c + e^n_c\|_{L^2} \| e^{n+1}_c+e^{n-1}_c \|_{L^2}  \notag \\
	\leq & C \tau h^4 + C\tau \|e^n_c\|^2_{L^2} + \epsilon_3 \tau \| \nabla(\e^{n+1}_\u+\e^{n-1}_\u)\|^2_{L^2}+\epsilon_4 \tau \| \nabla(e^{n+1}_c+e^{n-1}_c)\|^2_{L^2}.\notag
\end{align}

Finally, substituting the above inequalities into (\ref{biocnlf-errorequation2}), for sufficiently small $\epsilon_3, \epsilon_4$, we can derive 
\begin{align}
	\| \e&^{n+1}_\u\|^2_{L^2}+\| e^{n+1}_c\|^2_{L^2}-\| \e^{n-1}_\u\|^2_{L^2}-\| e^{n-1}_c\|^2_{L^2}\\
	&+ \kappa \tau \| \nabla (\e^{n+1}_\u+\e^{n+1}_\u)\|^2_{L^2}+ \theta\tau \|   \nabla (e^{n+1}_c + e^{n-1} _c)\|^2_{L^2} \notag\\
	\leq & C \tau h^4 + C \tau (\| e^{n+1}_c\|^2_{L^2}+\| e^{n-1}_c\|^2_{L^2}+ \|e^n_c\|_{L^2})
+C \tau (\| \e^{n+1}_\u\|^2_{L^2}+\| \e^{n-1}_\u\|^2_{L^2}+ \|\e^n_\u\|_{L^2})\notag\\
&+C \tau\big(\left\| \frac{\u^{n+1} - \u^{n-1}}{2\tau} - \u_t(t_n) \right\|_{L^2}^2+\left\| \frac{c^{n+1} - c^{n-1}}{2\tau} - c_t(t_n) \right\|_{L^2}^2\big)   \notag\\
&+C \tau \big( \| \nabla (\u^n - \frac{\u^{n+1} + \u^{n-1})}{2}\|^2_{L^2}+| \nabla ( \frac{ c^{n+1}+ c^{n-1}}{2} -c^n) \|^2_{L^2}\big)\notag\\
& + C \big(\int_{\Omega} \int_{t^{n-1}}^{t^{n+1}} |\eta_{\u t}|^2 dt dx+ \int_{\Omega} \int_{t^{n-1}}^{t^{n+1}} |\eta_{c t}|^2 dt dx \big)\notag.
\end{align}

Summing from $n=1$ to $n=N-1$, we get
\begin{align}
		\| \e&^{n+1}_\u\|^2_{L^2}+\| \e^{n}_\u\|^2_{L^2}+\|e^{n+1}_c\|^2_{L^2}+\| e^{n}_c\|^2_{L^2}\\
	&+ \kappa \tau \sum_{n=1}^{N-1}\| \nabla (\e^{n+1}_\u+\e^{n+1}_\u)\|^2_{L^2}+ \theta\tau\sum_{n=1}^{N-1} \|   \nabla (e^{n+1}_c + e^{n-1} _c)\|^2_{L^2} \notag\\
	\leq& C(\tau^4+h^4)+\big(\|\e^1_\u\|_{L^2}+\|\e^0_\u\|_{L^2} +\|e^1_c\|_{L^2}+\|e^0_c\|_{L^2}\big)+C\tau\sum_{n=1}^{N-1}\big(e^{n+1}_\u\|^2_{L^2}+\| e^{n+1}_c\|^2_{L^2} \big)\notag
\end{align}

Appplying Gronwall inequalities in Lemma \ref{biobdf-11}, we complete the proof of Theorem \ref{biocnlf-theorem1}.\end{proof}

Finally, according the projection error (\ref{biocnlf-projection-u}), (\ref{biocnlf-projection-c}) and Theorem \ref{biocnlf-theorem1}, we can derive
\begin{align}
	\| \mathbf{u}^i - \mathbf{u}_h^i \|_{L^2} = \| \mathbf{u}^i - \P^{n+1}_h \mathbf{u}^i + \P^{n+1}_h \mathbf{u}^i - \mathbf{u}_h^i \|_{L^2} \leq \| \mathbf{u}^i - \P^{n+1}_h \mathbf{u}^i \|_{L^2} + \| e_{\mathbf{u}}^i \|_{L^2} \leq C(\tau^4 + h^4).
\end{align}

Similarly
\begin{align}
	\| c^i - c_h^i \|_{L^2} = \| c^i - R^{n+1}_h c^i + R^{n+1}_h c^i - c_h^i \|_{L^2} \leq \| c^i - R^{n+1}_h c^i \|_{L^2} + \| e_c^i \|_{L^2} \leq C(\tau^4 + h^4).
\end{align}

The proof of Theorem \ref{biocnlf-theorem2} is done.
	\section{Numerical experience}\label{sec-5}
In this section, to validate the theoretical results of the fully discrete CNLF scheme (\ref{biocnlf-1})-(\ref{biocnlf-3}), we carry out numerical experiments using the software package FreeFEM++ \cite{hecht2012}. In particular, a test problem with an exact analytical solution is employed to confirm both the stability and the optimal convergence rate of the proposed method.

We consider the computational domain $\Omega = [0,1] \times [0,1]$ with the parameters $\theta=\gamma = 1$ and the final time $T = 1.0$. 
The exact solutions are taken from \cite{lichenyangbio}. 
\begin{align}\left\{\begin{aligned}
		&\u(x,y,t)=(yexp(-t)(2y - 1)(y - 1),  -xexp(-t)(2x - 1)(x - 1))^{T},\\
		&p(x,y,t)= exp(-t)(2x - 1)(2y - 1),\\
		& c(x,y,t)=exp(-t)sin(\pi x)sin(\pi y),
	\end{aligned}\right.\end{align}
Denote
\begin{align*}
	\|r - r_h\|_{L^2} &= \|r(t_N) - r_h^N\|_{L^2}, \\
	\|\mathbf{v} - \mathbf{v}_h\|_{L^2} &= \|\mathbf{v}(t_N) - \mathbf{v}_h^N\|_{L^2}, \\
	\|q - q_h\|_{l_2(L^2)} &= \left(\tau \sum_{n=1}^{N} \|q(t_n) - q_h^n\|_{L^2}^2\right)^{1/2}.
\end{align*}

The Taylor--Hood element (P1b--P1) is employed to approximate the velocity and pressure fields, while a linear Lagrange element is used for the concentration variable. 
The true solution is adopted to initialize the first time step, and the backward Euler scheme is applied to compute $\mathbf{u}_h^1$. 
	\begin{figure}
	\centering
	\includegraphics[width=0.6\textwidth]{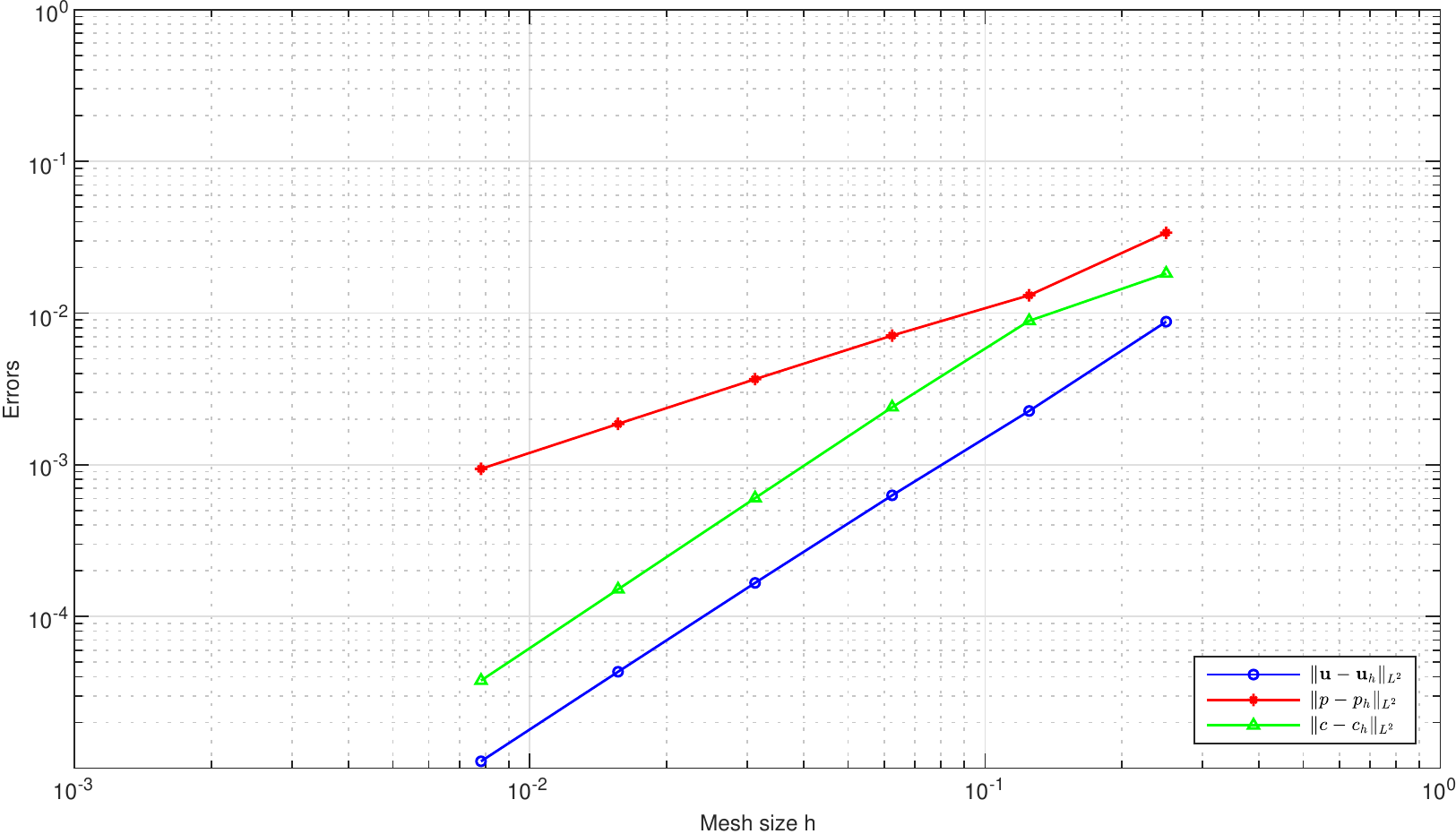}
	\caption{Convergence history of $(\u,p,c)$ for $\nu=1$.}
	\label{v1}
\end{figure}
\begin{figure}
	\centering
	\includegraphics[width=0.6\textwidth]{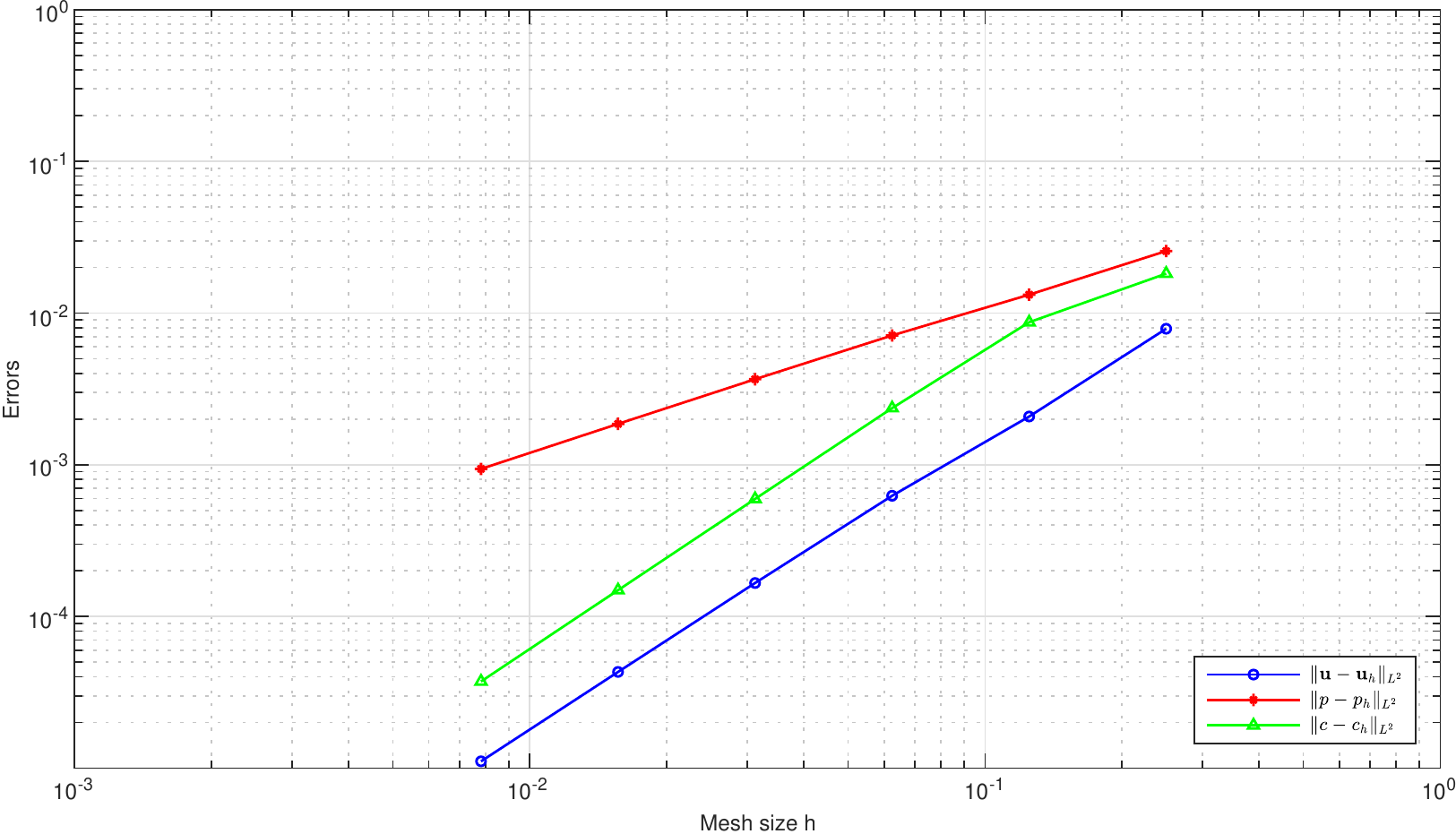}
	\caption{Convergence history of $(\u,p,c)$ for $\nu=1+0.1c$.}
	\label{v2}
\end{figure}
\begin{figure}
	\centering
	\includegraphics[width=0.6\textwidth]{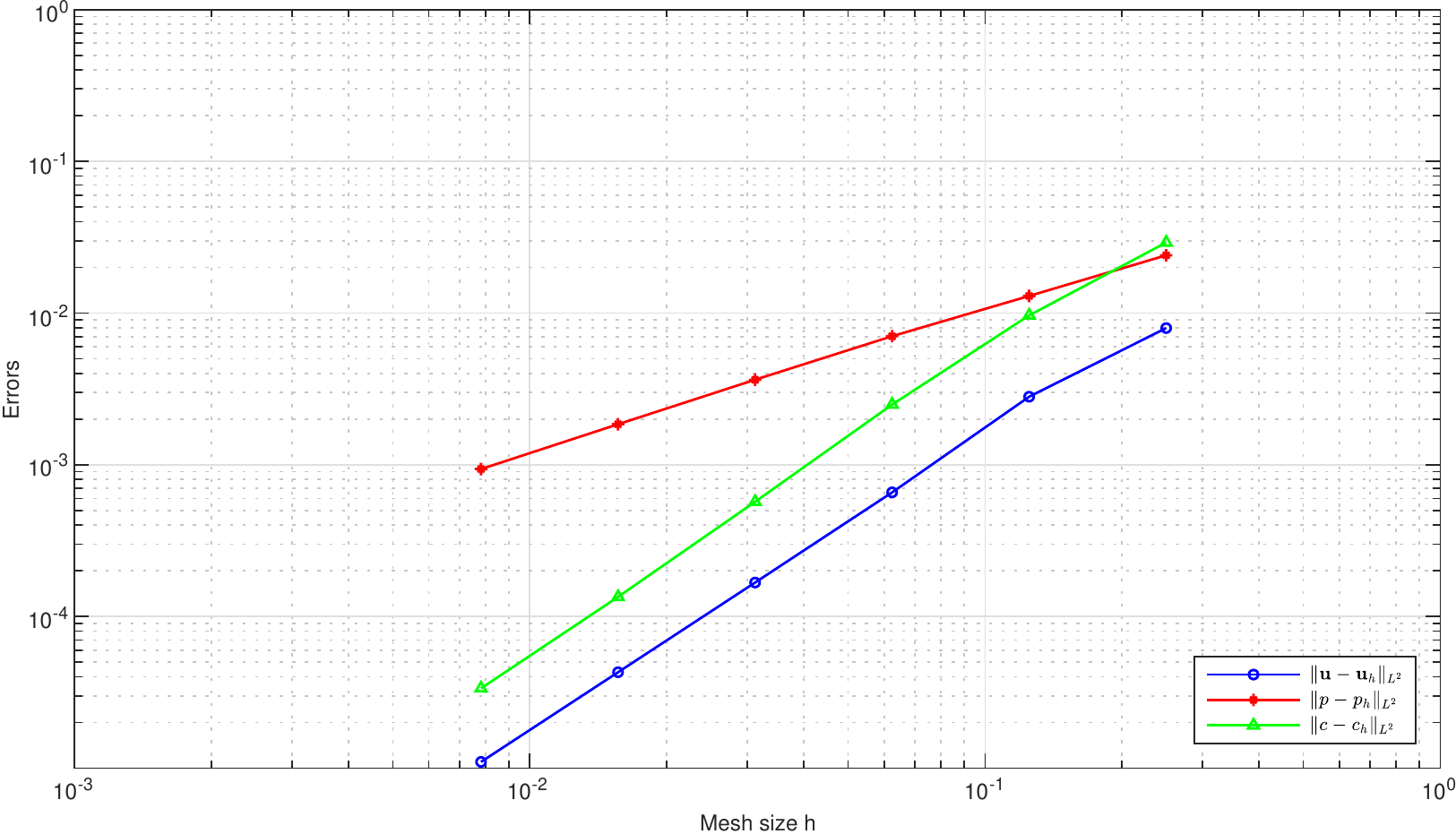}
	\caption{Convergence history of $(\u,p,c)$ for $\nu=exp(c)$.}
	\label{v3}
\end{figure}
To investigate the influence of different viscosity laws, we consider the following three cases:
\begin{equation*}
	\nu = 1, \qquad 
	\nu = 1 + 0.1c, \qquad 
	\nu = e^{c}.
\end{equation*}

To test the stability of the method, we present the stability results for different viscosity in Table \ref{stabilityv=1}, \ref{stabilityv=1+0.1c}, \ref{stabilityv=exp(c)}. To demonstrate the convergence behavior of the proposed scheme, we set the time step size as $\tau = h$ and refine the spatial mesh with
$
h = 1/4,\ 1/8,\ 1/16,\ 1/32,\ 1/64,\ 1/128.
$
The errors between the numerical and exact solutions for all variables at the final time $T = 1$ under different mesh sizes are summarized in Tables~\ref{l2v=1}, \ref{l2v=1+0.1c}, \ref{l2v=exp(c)} 
and illustrated graphically in Figure \ref{v1},\ref{v2},\ref{v3}.
It can be observed that the proposed fully discrete scheme achieves an optimal second-order accuracy in the $L^2$-norm for both the velocity $\mathbf{u}$ and concentration $c$. 
The pressure $p$ exhibits first-order accuracy in the $L^2$-norm, which is consistent with the theoretical prediction.

Furthermore, as shown in Tables~\ref{H1v=1},\ref{H1v=1+0.1c},\ref{H1v=exp(c)}, the velocity and concentration also exhibit first-order convergence in the $H^1$-norm.
The corresponding relative errors are reported in Tables~\ref{relativev=1},\ref{relativev=1+0.1c},\ref{relativev=exp(c)}, which again confirm the second-order accuracy of the proposed method in the $L^2$-norm for $\mathbf{u}$ and $c$, and first-order accuracy for the pressure variable. 

Overall, the numerical results validate the theoretical convergence rates and demonstrate the robustness of the proposed algorithm for different viscosity.

\begin{table}[htbp]  
	\centering  
	\caption{Stability result with $\nu=1$}   
	 	\setlength{\tabcolsep}{4.3mm}{ 
		\rowcolors{1}{gray!20}{white} 
 \begin{tabular}{cccccc}         
 	 $\tau=h$& $\| \u_h\|_{L^2} $   & $\| \u_h\|_{H^1}$   & $\| c_h\|_{L^2} $  & $\| c_h\|_{H^1}$   & $\| p_h\|_{L^2} $   \\  
 	   1/4   & 0.0388542 & 0.313623 & 0.171371 & 0.820493 & 0.136281 \\    1/8   & 0.0354237 & 0.159925 & 0.175669 & 0.795545 & 0.13232 \\    1/16  & 0.0355727 & 0.0805211 & 0.181718 & 0.811241 & 0.128778 \\    1/32  & 0.0358079 & 0.0404005 & 0.183381 & 0.815722 & 0.126026 \\    1/64  & 0.035876 & 0.0202365 & 0.1838 & 0.816847 & 0.124412 \\    1/128 & 0.0358946 & 0.0101282 & 0.183905 & 0.817129 & 0.123548 \\    \end{tabular} }
	 \label{stabilityv=1}
\end{table}%

	\begin{table}[htbp]  
		\centering  
		\caption{Numerical errors and convergence rates of $(\u,p,c)$  in $L^2$-norm with $\nu=1$}   
			 	\setlength{\tabcolsep}{3.8mm}{ 
			\rowcolors{1}{gray!20}{white}  \begin{tabular}{ccccccc}          
			$\tau=h$&  $\|\u -\u_h\|_{L^2}$ & rate  &$\| c -c_h\|_{L^2}$ & rate  & $\| p -p_h\|_{L^2}$   & rate \\   
			 1/4   & 0.0087769 &       & 0.0182156 &       & 0.033836 &  \\  
			   1/8   & 0.002263 & 1.96  & 0.0088862 & 1.04  & 0.0130976 & 1.37  \\  
			     1/16  & 0.0006286 & 1.85  & 0.002394 & 1.89  & 0.0071104 & 0.88  \\   
			      1/32  & 0.0001664 & 1.92  & 0.000603 & 1.99  & 0.0036656 & 0.96  \\    
			      1/64  & 4.32E-05 & 1.95  & 0.000151 & 2.00  & 1.86E-03 & 0.98  \\   
			       1/128 & 1.11E-05 & 1.96  & 3.78E-05 & 2.00  & 9.40E-04 & 0.98  \\    \end{tabular} }
			        \label{l2v=1}
	\end{table}%

		\begin{table}[htbp]  \centering  \caption{Numerical errors and convergence rates of $(\u,p,c)$ in $H^1$-norm with $\nu=1$}   
				 	\setlength{\tabcolsep}{7.3mm}{ 
				\rowcolors{1}{gray!20}{white} 
			 \begin{tabular}{ccccc}          
			 	$\tau=h$	& $\|\u -\u_h\|_{H^1}$ & rate  & $\|c -c_h\|_{H^1}$& rate  \\  
				  1/4   & 0.411564 &       & 0.316306 &  \\    
				  1/8   & 0.209972 & 0.97  & 0.159082 & 0.99  \\   
				   1/16  & 0.10573 & 0.99  & 0.0800549 & 0.99  \\
				       1/32  & 0.053031 & 1.00  & 0.0400932 & 1.00  \\   
				        1/64  & 0.026556 & 1.00  & 0.0200549 & 1.00  \\   
				         1/128 & 0.0132893 & 1.00  & 0.0100285 & 1.00  \\    \end{tabular} }
				         \label{H1v=1}
				         \end{table}%

			\begin{table}[htbp]  
				\centering  
				\caption{Relative errors and convergence rates of $(\u,p,c)$ with $\nu=1$}    
					\setlength{\tabcolsep}{4.1mm}{ 
					\rowcolors{1}{gray!20}{white} 
				\begin{tabular}{ccccccc}      
					  $\tau=h$	& $\frac{\| \u-\u_h\|_{L^2}}{\| \u\|_{L^2}}$ & rate  &$\frac{\| c-c_h\|_{L^2}}{\| c \|_{L^2}}$   & rate  &$\frac{\| p-p_h\|_{L^2}}{\| p\|_{L^2}}$ &rate \\ 
					       1/4   & 0.225894 &       & 0.106294 &       & 0.24828 &  \\    1/8   & 0.0638837 & 1.82  & 0.0505849 & 1.07  & 0.0989841 & 1.33  \\   
					        1/16  & 0.0176721 & 1.85  & 0.0131745 & 1.94  & 0.0552144 & 0.84  \\   
					         1/32  & 0.0046483 & 1.93  & 0.0032885 & 2.00  & 0.0290859 & 0.92  \\  
					           1/64  & 0.0012037 & 1.95  & 0.0008215 & 2.00  & 0.014917 & 0.96  \\   
					            1/128 & 0.0003087 & 1.96  & 0.0002053 & 2.00  & 0.0076048 & 0.97  \\    \end{tabular} } \label{relativev=1}
				    \end{table}%

\begin{table}[htbp]  
	\centering  
	\caption{Stability result with $\nu=1+0.1c$}   
	\setlength{\tabcolsep}{4.3mm}{ 
		\rowcolors{1}{gray!20}{white} 
		\begin{tabular}{cccccc}         
			$\tau=h$& $\| \u_h\|_{L^2} $   & $\| \u_h\|_{H^1}$   & $\| c_h\|_{L^2} $  & $\| c_h\|_{H^1}$   & $\| p_h\|_{L^2} $   \\  
		1/4   & 0.0386815 & 0.312889 & 0.171306 & 0.820245 & 0.134454 \\    1/8   & 0.0354475 & 0.159799 & 0.175853 & 0.796381 & 0.132311 \\    1/16  & 0.0355749 & 0.0805206 & 0.181751 & 0.81139 & 0.12877 \\    1/32  & 0.0358083 & 0.0404026 & 0.183389 & 0.815759 & 0.126023 \\    1/64  & 0.0358761 & 0.0202378 & 0.183802 & 0.816856 & 0.124411 \\    1/128 & 0.0358946 & 0.0101289 & 0.183905 & 0.817131 & 0.123548 \\    \end{tabular} }
	\label{stabilityv=1+0.1c}
\end{table}%
%

\begin{table}[htbp]  
	\centering  
	\caption{Numerical errors and convergence rates of $(\u,p,c)$  in $L^2$-norm with $\nu=1+0.1c$}   
	\setlength{\tabcolsep}{3.8mm}{ 
		\rowcolors{1}{gray!20}{white}  \begin{tabular}{ccccccc}          
			$\tau=h$&  $\|\u -\u_h\|_{L^2}$ & rate  &$\| c -c_h\|_{L^2}$ & rate  & $\| p -p_h\|_{L^2}$   & rate \\   
1/4   & 0.0078822 &       & 0.0181978 &       & 0.0256366 &  \\    1/8   & 0.0020801 & 1.92  & 0.0087013 & 1.06  & 0.0132166 & 0.96  \\    1/16  & 0.0006247 & 1.74  & 0.0023634 & 1.88  & 0.0071215 & 0.89  \\    1/32  & 0.0001662 & 1.91  & 0.0005956 & 1.99  & 0.0036635 & 0.96  \\    1/64  & 4.31E-05 & 1.95  & 0.0001491 & 2.00  & 1.86E-03 & 0.98  \\    1/128 & 1.11E-05 & 1.96  & 3.73E-05 & 2.00  & 9.39E-04 & 0.98  \\    \end{tabular} }
	\label{l2v=1+0.1c}
\end{table}%

\begin{table}[htbp]  \centering  \caption{Numerical errors and convergence rates of $(\u,p,c)$ in $H^1$-norm with $\nu=1+0.1c$}   
	\setlength{\tabcolsep}{7.3mm}{ 
		\rowcolors{1}{gray!20}{white} 
		\begin{tabular}{ccccc}          
			$\tau=h$	& $\|\u -\u_h\|_{H^1}$ & rate  & $\|c -c_h\|_{H^1}$& rate  \\  
1/4   & 0.410592 &       & 0.316334 &  \\    1/8   & 0.209804 & 0.97  & 0.159045 & 0.99  \\    1/16  & 0.105729 & 0.99  & 0.080051 & 0.99  \\    1/32  & 0.0530339 & 1.00  & 0.0400927 & 1.00  \\    1/64  & 0.0265578 & 1.00  & 0.0200548 & 1.00  \\    1/128 & 0.0132903 & 1.00  & 0.0100285 & 1.00  \\     \end{tabular} }
	\label{H1v=1+0.1c}
\end{table}%

\begin{table}[htbp]  
	\centering  
	\caption{Relative errors and convergence rates of $(\u,p,c)$ with $\nu=1+0.1c$}    
	\setlength{\tabcolsep}{4.1mm}{ 
		\rowcolors{1}{gray!20}{white} 
		\begin{tabular}{ccccccc}      
			$\tau=h$	& $\frac{\| \u-\u_h\|_{L^2}}{\| \u\|_{L^2}}$ & rate  &$\frac{\| c-c_h\|_{L^2}}{\| c \|_{L^2}}$   & rate  &$\frac{\| p-p_h\|_{L^2}}{\| p\|_{L^2}}$ &rate \\ 
	  1/4   & 0.203772 &       & 0.10623 &       & 0.190672 &  \\    1/8   & 0.0586805 & 1.80  & 0.0494804 & 1.10  & 0.09989 & 0.93  \\    1/16  & 0.0175609 & 1.74  & 0.0130032 & 1.93  & 0.055304 & 0.85  \\    1/32  & 0.0046412 & 1.92  & 0.0032475 & 2.00  & 0.0290699 & 0.93  \\    1/64  & 0.0012022 & 1.95  & 0.0008113 & 2.00  & 0.0149105 & 0.96  \\    1/128 & 0.0003084 & 1.96  & 0.0002028 & 2.00  & 0.0076029 & 0.97  \\   \end{tabular} } \label{relativev=1+0.1c}
\end{table}%

\begin{table}[htbp]  
	\centering  
	\caption{Stability result with $\nu=exp(c)$}   
	\setlength{\tabcolsep}{4.3mm}{ 
		\rowcolors{1}{gray!20}{white} 
		\begin{tabular}{cccccc}         
			$\tau=h$& $\| \u_h\|_{L^2} $   & $\| \u_h\|_{H^1}$   & $\| c_h\|_{L^2} $  & $\| c_h\|_{H^1}$   & $\| p_h\|_{L^2} $   \\  
		 1/4   & 0.0363009 & 0.311591 & 0.157871 & 0.759323 & 0.134105 \\    1/8   & 0.0345904 & 0.159779 & 0.175826 & 0.7974 & 0.132083 \\    1/16  & 0.0355937 & 0.0804668 & 0.182137 & 0.813323 & 0.128678 \\    1/32  & 0.0358143 & 0.0403885 & 0.183464 & 0.816105 & 0.125996 \\    1/64  & 0.0358771 & 0.0202471 & 0.183819 & 0.816931 & 0.124403 \\    1/128 & 0.0358948 & 0.0101343 & 0.183909 & 0.81715 & 0.123545 \\ 
	     \end{tabular} }
	\label{stabilityv=exp(c)}
\end{table}%
%

\begin{table}[htbp]  
	\centering  
	\caption{Numerical errors and convergence rates of $(\u,p,c)$  in $L^2$-norm with $\nu=exp(c)$}   
	\setlength{\tabcolsep}{3.8mm}{ 
		\rowcolors{1}{gray!20}{white}  \begin{tabular}{ccccccc}          
			$\tau=h$&  $\|\u -\u_h\|_{L^2}$ & rate  &$\| c -c_h\|_{L^2}$ & rate  & $\| p -p_h\|_{L^2}$   & rate \\   
		 1/4   & 0.007963 &       & 0.0291921 &       & 0.0240411 &  \\    1/8   & 0.002809 & 1.50  & 0.0096519 & 1.60  & 0.0129657 & 0.89  \\    1/16  & 0.0006577 & 2.09  & 0.0024943 & 1.95  & 0.0070393 & 0.88  \\    1/32  & 0.0001674 & 1.97  & 0.0005704 & 2.13  & 0.0036358 & 0.95  \\    1/64  & 4.29E-05 & 1.97  & 0.0001347 & 2.08  & 1.85E-03 & 0.98  \\    1/128 & 1.10E-05 & 1.97  & 3.36E-05 & 2.00  & 9.37E-04 & 0.98  \\    \end{tabular} }
	\label{l2v=exp(c)}
\end{table}%

\begin{table}[htbp]  \centering  \caption{Numerical errors and convergence rates of $(\u,p,c)$ in $H^1$-norm with $\nu=exp(c)$}   
	\setlength{\tabcolsep}{7.3mm}{ 
		\rowcolors{1}{gray!20}{white} 
		\begin{tabular}{ccccc}          
			$\tau=h$	& $\|\u -\u_h\|_{H^1}$ & rate  & $\|c -c_h\|_{H^1}$& rate  \\  
	   1/4   & 0.409354 &       & 0.318466 &  \\    1/8   & 0.210096 & 0.96  & 0.165691 & 0.94  \\    1/16  & 0.105655 & 0.99  & 0.0823851 & 1.01  \\    1/32  & 0.0530387 & 0.99  & 0.0403928 & 1.03  \\    1/64  & 0.0265718 & 1.00  & 0.0200554 & 1.01  \\    1/128 & 0.0132977 & 1.00  & 0.0100285 & 1.00  \\      \end{tabular} }
	\label{H1v=exp(c)}
\end{table}%

\begin{table}[htbp]  
	\centering  
	\caption{Relative errors and convergence rates of $(\u,p,c)$ with $\nu=exp(c)$}    
	\setlength{\tabcolsep}{4.1mm}{ 
		\rowcolors{1}{gray!20}{white} 
		\begin{tabular}{ccccccc}      
			$\tau=h$	& $\frac{\| \u-\u_h\|_{L^2}}{\| \u\|_{L^2}}$ & rate  &$\frac{\| c-c_h\|_{L^2}}{\| c \|_{L^2}}$   & rate  &$\frac{\| p-p_h\|_{L^2}}{\| p\|_{L^2}}$ &rate \\ 
		1/4   & 0.219361 &       & 0.184911 &       & 0.17927 &  \\    1/8   & 0.0812066 & 1.43  & 0.0548949 & 1.75  & 0.0981636 & 0.87  \\    1/16  & 0.0184773 & 2.14  & 0.0136944 & 2.00  & 0.0547046 & 0.84  \\    1/32  & 0.0046728 & 1.98  & 0.0031089 & 2.14  & 0.0288567 & 0.92  \\    1/64  & 0.0011944 & 1.97  & 0.0007325 & 2.09  & 0.0148456 & 0.96  \\    1/128 & 0.0003052 & 1.97  & 0.0001829 & 2.00  & 0.0075843 & 0.97  \\  \end{tabular} }
		 \label{relativev=exp(c)}
\end{table}%

\section{Conclusion}\label{sec-6}
In this paper, we present the unconditional stability and error estimates of the decoupled Crank--Nicolson Leap--Frog (CNLF) method for solving unsteady bioconvection flows with concentration-dependent viscosity. Numerical results are provided for a test problem with an analytical solution, demonstrating that the decoupled CNLF method performs robustly. The numerical experiments indicate that the method achieves second-order accuracy in both time and space for the $L^2$-norms of $\mathbf{u}_h$ and $c_h$. Compared with the CNLE (or CNSLE) scheme, the CNLF method attains the same accuracy with improved efficiency. In future work, we plan to extend the CNLF framework to more complex bioconvection systems, including the Chemotaxis--Navier--Stokes system, the Patlak--Keller--Segel--Navier--Stokes system, and the Chemo--Repulsion--Navier--Stokes system.

\section*{CRediT authorship contribution statement}
Chenyang Li:
Writing -- original draft, Visualization, Validation, Software, Methodology, Conceptualization;

\section*{Data availability}
Data will be made available on request.
Declaration of competing interest
The authors declare that they have no known competing financial interests or personal relationships
that could have appeared to influence the work reported in this paper.
\section*{Acknowledgments}
The authors would like to thank the editor and referees for their valuable comments and suggestions
which helped us to improve the results of this paper.
\bibliographystyle{abbrv}
	

\end{document}